\documentclass[11pt,reqno,a4paper]{amsart}

\usepackage[margin=1.05in,footskip=0.15in,tmargin=1in,bmargin=0.85in]{geometry}
\usepackage[indent,skip=.2\baselineskip]{parskip}
\usepackage[utf8]{inputenc}
\usepackage[english]{babel}
\usepackage{amsthm,amsmath,amssymb}
\usepackage{mathtools}
\usepackage{amsfonts}
\usepackage{enumitem}
\usepackage{graphicx}
\usepackage[font=footnotesize]{caption}
\usepackage{tikz-cd}
\usepackage{array}
\usepackage{fancyhdr}

\usepackage{comment}
\usepackage{aliascnt}
 \usepackage[backend=bibtex,style=alphabetic,natbib=true,
 doi=false,isbn=false,url=false,eprint=false,maxbibnames=9]{biblatex}

\usepackage[unicode,psdextra,pdfusetitle,
bookmarks, bookmarksdepth=2, colorlinks=true,
linkcolor=blue!50!black, citecolor=teal, urlcolor=blue!50!black]{hyperref}
\usepackage[capitalise, noabbrev,nameinlink]{cleveref}

\newcommand{\PP}{\mathbb{P}}

\newcommand{\R}{\mathbb{R}}
\newcommand{\C}{\mathbb{C}}

\AtEveryBibitem{%
  \clearfield{month}%
  \iffieldundef{eprint}{}{\clearfield{year}\clearfield{date}}%
}

\newtheorem{theorem}{Theorem}[section]

\newaliascnt{proposition}{theorem}
\newtheorem{proposition}[proposition]{Proposition}
\aliascntresetthe{proposition}
\newaliascnt{lemma}{theorem}
\newtheorem{lemma}[lemma]{Lemma}
\aliascntresetthe{lemma}
\newaliascnt{corollary}{theorem}
\newtheorem{corollary}[corollary]{Corollary}
\aliascntresetthe{corollary}

\theoremstyle{definition}

\newaliascnt{definition}{theorem}
\newtheorem{definition}[definition]{Definition}
\aliascntresetthe{definition}
\newaliascnt{example}{theorem}
\newtheorem{example}[example]{Example}
\aliascntresetthe{example}
\newtheorem{example*}{Example}
\newaliascnt{remark}{theorem}
\newtheorem{remark}[remark]{Remark}
\aliascntresetthe{remark}
\newaliascnt{conjecture}{theorem}

\aliascntresetthe{conjecture}
\newaliascnt{question}{theorem}

\aliascntresetthe{question}
\newaliascnt{assumption}{theorem}

\aliascntresetthe{assumption}
\newaliascnt{notation}{theorem}

\aliascntresetthe{notation}

\crefname{theorem}{theorem}{theorems}
\crefname{construction}{construction}{constructions}
\crefname{proposition}{proposition}{propositions}
\crefname{lemma}{lemma}{lemmas}
\crefname{corollary}{corollary}{corollaries}
\crefname{definition}{definition}{definitions}
\crefname{example}{example}{examples}
\crefname{remark}{remark}{remarks}
\crefname{conjecture}{conjecture}{conjectures}
\crefname{question}{question}{questions}
\crefname{assumption}{assumption}{assumptions}
\crefname{notation}{notation}{notations}
\Crefname{theorem}{Theorem}{Theorems}
\Crefname{construction}{Construction}{Constructions}
\Crefname{proposition}{Proposition}{Propositions}
\Crefname{lemma}{Lemma}{Lemmas}
\Crefname{corollary}{Corollary}{Corollaries}
\Crefname{definition}{Definition}{Definitions}
\Crefname{example}{Example}{Examples}
\Crefname{remark}{Remark}{Remarks}
\Crefname{conjecture}{Conjecture}{Conjectures}
\Crefname{question}{Question}{Questions}
\Crefname{assumption}{Assumption}{Assumptions}
\Crefname{notation}{Notation}{Notations}

\newtheorem*{theorem*}{Theorem}
\newtheorem*{proposition*}{Proposition}
\newtheorem*{definition*}{Definition}
\newtheorem{theoremalphabetic}{Theorem}

\newenvironment{thma}
  {\setcounter{theoremalphabetic}{0}\begin{theoremalphabetic}}
  {\end{theoremalphabetic}}

\newenvironment{thmb}
  {\setcounter{theoremalphabetic}{1}\begin{theoremalphabetic}}
  {\end{theoremalphabetic}}

\newenvironment{thmc}
  {\setcounter{theoremalphabetic}{2}\begin{theoremalphabetic}}
  {\end{theoremalphabetic}}

\crefname{theoremalphabetic}{theorem}{theorems}
\Crefname{theoremalphabetic}{Theorem}{Theorems}

\DeclareMathOperator{\rank}{rank}

\DeclareMathOperator{\supp}{supp}

\DeclareMathOperator{\diag}{diag}
\DeclareMathOperator{\Res}{Res}

\DeclareMathOperator{\spann}{span}

\DeclareMathOperator{\rk}{rk}

\numberwithin{equation}{section}

\DeclareTextFontCommand{\bfemph}{\bfseries\em}

\newcommand{\RR}{\mathbb{R}}
\newcommand{\Tn}{\mathbb{T}^n}
\newcommand{\MV}{\text{MV}}
\newcommand{\ZZ}{\mathbb{Z}}

\newcommand{\Id}{\textrm{Id}}

\newcommand{\conv}{\textrm{conv}}

\newcommand{\CC}{\mathbb{C}}

\newcommand{\cl}{\text{cl}}
\newcommand{\B}{\mathcal{B}}

\DeclareMathOperator{\Trop}{Trop}

\DeclareMathOperator{\mC}{\mathcal{C}}

\DeclareMathOperator{\codim}{codim}

\title{Combinatorics of degeneracy in vertical systems}

\author{Carles Checa and Nidhi Kaihnsa}

\date{\today}

\providecommand{keywords}[1]{
  \small
  \textbf{\textit{Keywords: }} #1
}

\newcommand{\Addresses}{{%
  \bigskip
  \footnotesize
  \textsc{Carles Checa, University of Copenhagen}\par\nopagebreak
  \textit{E-mail address}, \texttt{ccn@math.ku.dk}

  \medskip

 \textsc{Nidhi Kaihnsa, University of Copenhagen}\par\nopagebreak
  \textit{E-mail address},  \texttt{nidhi@math.ku.dk}
}}

\begin{document}

\begin{abstract}
We characterize the generic consistency of vertical systems in terms of the iterated circuits of the column matroid of
the coefficient matrix. More precisely, we prove that such a system is generically consistent if
and only if there exists an iterated circuit such that the dimension of the Minkowski sum of the associated Newton polytopes is equal to the rank of the coefficient matrix. We also give an
equivalent characterization in terms of $G$-spanning flags of the dual matroid. With this, we
derive a combinatorial formula for the codimension of the locus of parameter
values for which the system has a solution, extending known results on dual
defects of toric varieties. When this locus has codimension one, its defining
equation yields a generalization of the sparse resultant to vertical systems, which applies also to $\mathcal{A}$-discriminants.
Finally, we characterize when the solution set is contained, for generic
parameter values, in a finite union of hyperplanes parallel to a coordinate
hyperplane.
\end{abstract}

\maketitle

\section{Introduction}
\label{section:1}

Polynomial systems arising in a wide range of applications are rarely dense. Very often, the data encoded by their coefficients is supported on a sparse set of monomials, allowing many geometric properties to be studied through the combinatorics of Newton polytopes. At the same time, one is frequently interested in generic properties of these systems, that is, properties that persist under sufficiently small perturbations of the data.

Notable examples exploiting sparsity include the \emph{BKK bound} \cite[Chapter 7]{coxlitosh} for enumerating nonzero complex solutions via mixed volumes, as well as formulas for computing sparse resultants \cite{gkz1994, dandrea2020cannyemiris}. However, in many parametric families appearing in different applications, the same parameters occur in several equations, so the coefficients satisfy algebraic dependencies \cite{feliu2025genericgeometrysteadystate,genericFlatness, mixeddiscriminants}, leaving room for a more detailed study of the generic geometry of those parametric families.

Khovanskii characterized sparse systems for which a generic choice of the
coefficients gives rise to a complete intersection in $(\mathbb{C}^*)^n$
\cite[Theorem 11]{Khovanskii2016NewtonPA}. Following \cite{FELIU2025630},
we call this property
\emph{generic consistency}: for general parameter values, its solution set is
nonempty and has the expected dimension. This property can be characterized
through degeneracy: a system is generically
consistent if and only if there exists an open subset of the parameter space where all solutions are nondegenerate \cite[Theorem 3.7]{FELIU2025630}.  Our purpose is to
express this property combinatorially for other parametric families of polynomial systems.

\subsection*{Vertical systems. } Our main goal is to extend Khovanskii's theorem from sparse systems with independently varying coefficients to \emph{vertically parametrized systems} (or vertical systems, for short) \cite{genericFlatness, FELIU2025630}. These are polynomial systems of the form
	\begin{equation}
		\label{verticalSystem}
		F = C \big(u \star x^B \big) \ \in \CC[u,x^{\pm}]^{s}
	\end{equation}
with parameters $u=(u_1,\dots,u_m)$ and variables $x=(x_1,\dots,x_n)$. Here, $C\in\C^{s\times m}$ has rank $s$, and the columns $b_1,\dots,b_m$ of $B\in\mathbb{Z}^{n\times m}$ are the exponent vectors of the monomials. The notation $u\star x^B$ denotes the Hadamard product: the monomial $x^{b_i}$ is scaled by the parameter $u_i$. Thus, the same parameter $u_i$ accompanies the same monomial in every equation in which that monomial occurs. We study the vanishing locus over the torus $\Tn \coloneqq (\mathbb{C}^*)^n$ for generic values of $u\in\mathbb{C}^m$:
	$$V_{\Tn}(F_u) = \{x \in \Tn : \: C(u \star x^B) = 0\} \subset \Tn$$
	where $F_u$ is the polynomial system with coefficients specialized at $u$.

    A general framework for studying the generic geometry of vertical systems was developed in \cite{FELIU2025630, feliu2025genericgeometrysteadystate} and subsequently studied in \cite{genericFlatness,helmincktropical,feliu2025toricinvarianceverticallyparametrized,feliu2026rootboundsverticalsystems, checa2026toricvectorbundlesnumber, esterov2024engineeredcompleteintersectionsslightly, esterov2026engineeredcompleteintersectionsalgorithmic, kaveh2025vectorvaluedlaurentpolynomialequations}, including the study of generic root counts and numerical pipelines to compute with them.

	Vertical systems arise naturally in applications. Many models in applied algebraic geometry involving the critical equations of certain functions contain vertical systems or can be embedded into them. This includes several types of discriminants \cite{tropicaldiscriminants, mixeddiscriminants, iterateddiscriminants}. Another important source comes from biochemistry: the equations defining the steady states of chemical reaction networks have this form \cite{feinberg-book, feliu2025genericgeometrysteadystate}. Moreover, sparse systems with independently varying coefficients form a special case of vertical systems. Thus, the study of vertical systems is a natural generalization of sparse elimination theory \cite{sturmfels94, dickenstein2002elimination, khovanskiitoroidal}, including the study of $\mathcal{A}$-discriminants \cite{cattani2022nonsplittingflagsiteratedcircuits, CURRAN2007115, gkz1994, diroccodefect}.

    Motivated by the reaction network theory, the authors of \cite{FELIU2025630} gave a linear algebra criterion for the generic consistency of $F=C(u\star x^B)$, showing that generic consistency is equivalent to the existence of $w\in\ker(C)$ such that the matrix $C\diag(w)B^{\top}$ has rank $s$. For polynomials with real coefficients, the same criterion controls the expected dimension over the real and positive zero sets \cite[Theorem 3.7]{FELIU2025630}. Our first  objective is to provide a combinatorial characterization, in terms of iterated circuits of $M(C)$ and their Newton polytopes.

    \subsection*{Resultants.} When a system is not generically consistent, a second natural problem is to study the parameter values which lead to a system with a solution. This problem can be addressed using resultants. Sparse resultants originate in the work of Gelfand, Kapranov and Zelevinsky \cite{gkz1994} and have become a central tool in computational algebraic geometry and polynomial system solving \cite{sturmfels94,cannyemiris93,dandrea2020cannyemiris,coxlitosh}. Their computation, particularly through determinantal formulas, continues to be an active area of research \cite{dandrea2020cannyemiris,dandrea2025sparse,dandrea2002,cannypedersen,mantzaflaris:hal-01528377,checaemiris,bender2021koszultype,cannyemiris93}.

	Geometrically, the resultant is obtained from the incidence variety by projecting onto the coefficient space. For $n+1$ sparse polynomials in $n$ variables, Sturmfels described the codimension of this variety in terms of the dimensions of the Newton polytopes \cite[Theorem 1.1]{sturmfels94}. If its codimension equals one, the projection is defined by the sparse resultant. Our second objective is to derive such an analogue of this codimension formula for systems with vertical dependencies.

	\subsection*{Main results}
	Iterated circuits have appeared previously in the study of
$\mathcal{A}$-discriminants and the defectivity of toric varieties
\cite{esterov2010newton, 10.1007/s10801-018-0816-4,
cattani2022nonsplittingflagsiteratedcircuits}. Our first main contribution is
to relate the iterated circuits of the column matroid of $C$, together with
their associated Newton polytopes, to the generic consistency of vertical systems $F=C(u\star x^B)$. Recall that a circuit of $M(C)$ is a minimally dependent set of columns. An
\emph{iterated circuit} of $M(C)$ is a sequence
$(\mathcal{C}_1,\dots,\mathcal{C}_k)$ of subsets of $[m]$ such that
for each $i=0,\dots,k-1$, the set $\mathcal{C}_{i+1}$ is a circuit
    of the contraction of $M(C)$ by $\bigcup_{j=1}^i\mathcal{C}_j$. For every $X\subseteq[m]$, let
\[
    \Delta_X\coloneqq
    \operatorname{conv}\{b_j:j\in X\}\subseteq\mathbb{R}^n
\]
be the Newton polytope of the corresponding exponent vectors. We also give a formulation of this criterion in terms of its Gale dual system, by introducing \emph{$G$-spanning flags} (\Cref{def: nonsplitting_flag}), which generalize the non-splitting flags introduced in \cite{CURRAN2007115, cattani2022nonsplittingflagsiteratedcircuits}.

\begin{thma}
    (\Cref{thm:mainTheorem} and \Cref{thm: GspanningFlag})
    \label{thm:mainIntro}
    Let $C\in\mathbb{C}^{s\times m}$ have rank $s$, and let
    $B\in\mathbb{Z}^{n\times m}$. The following are equivalent:
    \begin{enumerate}[label=(\roman*)]
        \item The vertical system
        $F=C(u\star x^B)$ is generically consistent.
        \item There exists an iterated circuit
        $(\mathcal{C}_1,\dots,\mathcal{C}_k)$ of $M(C)$ such that
        \begin{equation}
            \label{eq: sumDELTAS}
            \dim\left(
                \sum_{j=1}^k\Delta_{\mathcal{C}_j}
            \right)=s.
        \end{equation}
        \item There exists a $G$-spanning flag of $M(C)^*$ of length $m-s$.
    \end{enumerate}
\end{thma}

In
particular, the existence of such a flag is equivalent to the generic consistency of the Gale dual system (\Cref{cor: GaleDual}).  In the case where the matroid of the columns of $C$ coincides with the matroid of the columns of $\overline{B}$, our notion of $G$-spanning flag coincides with the notion of non-splitting flag in \cite{CURRAN2007115, cattani2022nonsplittingflagsiteratedcircuits}. \Cref{ex: GaleDual} illustrates how generic consistency of the Gale dual system can be read from a flag of flats.

For square systems, \Cref{thm:mainIntro} also provides a combinatorial positivity criterion for the generic number of nonzero complex solutions. We show that the existence of an iterated circuit satisfying \eqref{eq: sumDELTAS} is equivalent to the positivity of the root-count formulas of \cite{genericFlatness,kaveh2025vectorvaluedlaurentpolynomialequations} (\Cref{thm: equivalentConditions}).

Unlike in the sparse case, one cannot independently rescale the equations of a vertical system. We prove, however, that the parameters may be projectivized separately along the connected components $\mathcal{B}_1,\dots,\mathcal{B }_c \subset [m]$ of $M(C)$ (\Cref{thm: connectedComponentsInvariant}). Let $\mathbb{P}^{\mathcal B} \coloneqq \mathbb{P}^{\mathcal B_1} \times \dots \times \mathbb{P}^{\mathcal B_c}$ and denote the resulting parameter space, let $\mathcal W_F\subseteq\Tn\times\mathbb{P}^{\mathcal B}$ be the incidence variety, and let $\pi$ be its projection onto $\mathbb{P}^{\mathcal B}$. We also set
\[
    \iota(F)\coloneqq
    \max_{(\mathcal C_1,\dots,\mathcal C_k)}
    \dim\left(\sum_{j=1}^k\Delta_{\mathcal C_j}\right),
\]
where the maximum is taken over the iterated circuits of $M(C)$. Finally, let $F_{(i)}$ denote the vertical system associated with contraction by the $i$-th column (\eqref{eq: systemFI}). The following theorem extends \cite[Theorem 1.1]{sturmfels94} from sparse systems with independent coefficients to vertical systems, as well as the characterization of when the projection of the incidence variety has codimension one.

	\begin{thmb}
		\label{thm:codimensionIntro}(\Cref{thm:mainCodimension} and \Cref{thm: codimensionOne})
		Let $C \in \mathbb{C}^{s \times m}$ be of rank $s$, $B \in \mathbb{Z}^{n \times m}$ and assume that $\mathcal{W}_F \neq \varnothing$.
		\begin{enumerate}[label=(\roman*)]
        \item The codimension of $\overline{\pi(\mathcal{W}_F)}$ equals $s - \iota(F)$.
			\item If $F$ is not generically consistent, then the codimension of $\overline{\pi(\mathcal{W}_F)}$ equals $1$ if and only if there exists $i \in [m]$ such that $F_{(i)}$ is generically consistent.
		\end{enumerate}
	\end{thmb}

	Thus, iterated circuits determine not merely whether a generic system has a solution, but the precise number of independent conditions imposed on the parameter space by the existence of a solution. If $\overline{\pi(\mathcal{W}_F)}$ has codimension one, it is defined by a unique homogeneous irreducible polynomial $\Res_F\in\CC[u]$, which we call the \emph{resultant of the vertical system}. We further determine which parameters can occur in $\Res_F$ in terms of generically consistent contractions (\Cref{corollary: parametersThatAppear}). This resultant can be computed in some simpler cases (\Cref{prop: sparseResultant} and \Cref{prop: lowcodimensionComputation}), but we leave more general formulas for further research.

    When the vertical system encodes the critical points of a fewnomial hypersurface, it has the form
    \begin{equation}
    \label{eq: FBintro}F_B = \overline{B}(u \star x^B)
    \end{equation}
    where the columns of $B \in \mathbb{Z}^{n \times m}$ encode the exponents of the monomials defining the hypersurface and $\overline{B} \in \mathbb{Z}^{(n+1) \times m}$ is formed from $B$ by inserting a row of ones at the top. In this case, the $\mathcal{A}$-discriminant corresponds to the resultant of \eqref{eq: FBintro}. We show that \Cref{thm:codimensionIntro} generalizes previous results stated in terms of  \textit{non-splitting flags} (see \Cref{def: nonsplitting}) \cite{cattani2022nonsplittingflagsiteratedcircuits, CURRAN2007115, diroccodefect}, which study when the locus of coefficients that give a fewnomial hypersurface with nonzero complex critical points has codimension one (\Cref{thm: dualDefects}).

As another application of the results above, we study when the variety $V_{\Tn}(F_u)$ is contained in a finite number of parallel hyperplanes in $\mathbb{C}^n$.
We say that a vertical system $F = C(u \star x^B)$ is \emph{linearly degenerate} with respect to $x_i$ if there exist $a_1,\dots,a_k \in \mathbb{C}^*$ such that $V_{\Tn}(F_{u}) \subset \bigcup_{j = 1}^k \{x_i = a_j\}$ for generic values of $u \in \mathbb{C}^m$. When restricted to positive $a_j$ and positive parameter values, this property is closely related to \emph{absolute concentration robustness} \cite{feliu2025genericgeometrysteadystate}, which arises in chemical reaction network theory, as well as to other properties such as sensitivity \cite{localglobal,precision,Kaihnsa_Nguyen_Shiu_2024}. In the case of a single possible value, it implies that the concentration corresponding to the $i$-th species remains invariant at every positive steady state for generic values of the parameters.

The work \cite{feliu2025genericgeometrysteadystate} relates linear degeneracy to the elimination ideal obtained by retaining only $x_i$ and the parameters. We combine this perspective with Theorems~\ref{thm:mainIntro} and~\ref{thm:codimensionIntro}. Let $\tau_i:\mathbb{R}^n\to\mathbb{R}^{n-1}$ be the coordinate projection that forgets the $i$-th coordinate.

\begin{thmc}(\Cref{thm: linearDegeneracy})
\label{thm: linearDegeneracyIntro}
Let $F = C(u \star x^B)$ be a generically consistent vertical system with $C$ being rank $s$. The following are equivalent:
    \begin{enumerate}[label=(\roman*)]
        \item The vertically parametrized family $F = C(u \star x^B)$ is linearly degenerate with respect to $x_i$.
        \item For every iterated circuit $(\mathcal{C}_1,\dots,\mathcal{C}_k)$ of $M(C)$,
        \[
            \dim\left(\sum_{j=1}^k\tau_i(\Delta_{\mathcal{C}_j})\right)<s.
        \]
    \item The polynomial $R_i$ defined in \eqref{eq:Ri} is nonzero and generates the elimination ideal $\langle F \rangle \cap \CC(u)[x_i^{\pm}]$.
    \end{enumerate}
\end{thmc}

The polynomial $R_i$ is defined by evaluating the resultant of the system after hiding the variable $x_i$. We show in examples that this resultant can be used to compute the values of $x_i$ when a vertical system is linearly degenerate with respect to $x_i$ (\Cref{ex: reactionNetwork}).

\subsection*{Structure of the paper.} In \Cref{section:2}, we recall some background on matroids needed
for our main results and relate it to basic properties of vertical
systems. We review generic consistency and projected systems and Gale dual systems. In \Cref{section:3}, we prove our main theorem
(\Cref{thm:mainTheorem}), which characterizes generic consistency in
terms of iterated circuits, and give an equivalent formulation using
$G$-spanning flags. We also establish the equivalence between
$G$-spanning and non-splitting flags under suitable hypotheses and provide a positivity criterion for the generic root counts of vertical systems.
In \Cref{section:4}, we prove the codimension formula
(\Cref{thm:mainCodimension}) for the closure of the locus of parameters
for which the system has a solution in $\Tn$. We recover known results on
the defectivity of toric varieties and define the resultant of a
vertical system when this closure is a hypersurface. Finally, we
compute this resultant in some special cases and use our criterion
to characterize linear degeneracy (\Cref{thm: linearDegeneracy}).

\subsection*{Acknowledgments}
This project has been funded by the European Union under the Grant Agreement number 101044561, POSALG. Views and opinions expressed are those of the authors only and do not necessarily reflect those of the European Union or European Research Council (ERC). Claude Opus 5 was used to address a gap in the proof of \Cref{thm:mainCodimension} and to assist in verifying the remaining proofs. All mathematical content was conceived and formulated by the authors, who take full responsibility for it.

The following notation will be used throughout the paper:

\begin{enumerate}[label=(\roman*)]
    \item $\Tn$ will denote the $n$-dimensional complex torus, i.e. $\Tn \coloneqq (\mathbb{C}^*)^n$.
    \item Given two vectors $u,v \in \CC^m$, $u \star v \in \mathbb{C}^m$ denotes the Hadamard product of the two vectors, i.e. the $i$-th entry equals $u_iv_i$.
    \item For any vector $w \in \mathbb{C}^m$, $\diag(w) \in \mathbb{C}^{m \times m}$ denotes the diagonal matrix formed with the entries of $w$.
    \item We denote by $\mathbf{1} \in \mathbb{R}^m$ the vector consisting of all ones.
    \item For $s$ vectors $v_1,\dots,v_s \in \mathbb{C}^n$, we denote by
	 $$\bigwedge_{i = 1}^s v_i = v_1 \wedge \dots \wedge v_{s} \in \wedge^{s}\mathbb{C}^n,$$
     their wedge product.
    \item We denote the ground set $[m] \coloneqq \{1,\dots,m\}$. For each subset $J \subset [m]$, we denote by $|J|$ its cardinality.
    \item If $J \subset [m]$, $\mathbb{P}^{J}$ denotes the $(|J| - 1)$-dimensional projective space over $\mathbb{C}$.
\end{enumerate}

\section{Combinatorics of vertical systems}
\label{section:2}

In this section, we define vertical systems and their properties and their relation to matroids. Moreover, we define iterated circuits and describe some of their properties with respect to the Newton polytopes of the equations of the system.

\subsection{Matroids} We briefly recall basic definitions for \textit{representable matroids} and for a detailed exposition, the readers are referred to \cite{oxley}. Given a matrix $C \in \mathbb{C}^{s \times m}$, let $M(C)$ be the matroid generated by the columns of $C$ with the rank function $$r_C(X) = \dim \spann \langle c_i  : \: i \in X\rangle.$$
\begin{enumerate}[label=(\roman*)]
     \item  A \emph{basis} of $M(C)$ is an independent set of maximal rank.
     \item
     A \textit{circuit} of $M(C)$ is a minimally dependent subset $\mathcal{C}\subset [m]$ of columns of $C$.

     \item A \emph{loop} is an element of $M(C)$ of rank $0$.
     \item A \emph{flat} $E$ is a subset $E \subset [m]$ such that
	 $r_{C}(E \cup \{x\}) > r_{C}(E)$
	 for all $x \notin E$. Given any set $J \subset [m]$, the closure of $J$, denoted by $\cl(J)$, is the smallest flat containing it.

     \item Given a subset $X\subset [m]$, the \textit{contraction} of $M(C)$ with respect to $X$, denoted by $M(X)$, is the matroid in the set of columns given by $[m]\setminus X$ with the rank function:
	 \begin{equation}
     \label{eq: rankFunction}
     r_{X}(Y) \coloneqq r_{C}(X \cup Y) - r_C(X).\end{equation}

We can describe a matrix representing the contraction \(M(X)\) as follows.
Let \(I\subseteq X\) be an independent set of maximal rank inside $X$. After reordering the columns and applying elementary row
operations, we may write \(C\) in row-reduced echelon form as
\[
\begin{pmatrix}
    \Id_r & 0 & P\\
    0 & \Id_{s-r} & \widetilde{Q}
\end{pmatrix},
\]
where the first \(r\) columns correspond to \(I\). Since \(I\) has maximal
rank in \(X\), the columns indexed by \(X\setminus I\) become loops after
contracting \(I\), i.e. they are $0$ columns in $\widetilde{Q}$. Thus, deleting the first \(r\) rows and the columns
indexed by \(X\) yields a matrix representing \(M(X)\):
\begin{equation}
\label{equation:contraction}
    C_X\coloneqq
    \begin{pmatrix}
        \Id_{s-r} & Q
    \end{pmatrix},
\end{equation}
where \(Q\) is obtained from \(\widetilde{Q}\) by deleting the
columns indexed by \(X\setminus I\).

 \item The \textit{restriction} of $M(C)$ to a subset $Y \subset [m]$, denoted by $M(C)_{\mid Y}$, is defined by restricting the rank function to $Y$. A matrix representing $M(C)_{\mid Y}$ can be computed by deleting the columns of $C$ that are not in $Y$.
      \item The \textit{dual matroid} $M(C)^*$ is the matroid with rank function
	 \begin{equation}
	    \label{rank:dual}
        r_C^*(X) = |X| + r_C([m] \setminus X) - s.
	 \end{equation}
     If $D \in \mathbb{C}^{m \times \ell}$ is a matrix whose columns generate the kernel of $C$, the dual matroid of $M(C)$ is represented by the rows of $D$, i.e. $M(C)^* = M(D)$, where $M(D)$ denotes the matroid of the rows of $D$.  By \cite[Proposition 3.1.6]{oxley}, for any subset $X \subset [m]$
     \begin{equation}
     \label{eq: contractionRestriction}
     M(X)^* = M(D)_{\mid [m] \setminus X}.
     \end{equation}
  A \textit{coloop} is a loop of the dual matroid.
     \item Two elements $e,e' \in [m]$ are connected in $M(C)$ if there exists a circuit $\mathcal{C} \subset [m]$ such that $e,e' \in \mathcal{C}$. This defines an equivalence relation whose equivalence classes are the connected components of $M(C)$.
     \end{enumerate}

Throughout, we denote by $\mathcal{B}_1,\dots,\mathcal{B}_c \subset [m]$ the connected components of $M(C)$. To find $\mathcal{B}_i$, we rearrange the columns of $C$, so that the row-reduced form of $C$ equals
	 \begin{equation}
	 	\label{eq: rref}
	 	C = \begin{pmatrix} \Id_s & P \end{pmatrix}
	 \end{equation}
	 for some matrix $P \in \mathbb{C}^{s \times \ell}$ where $\ell \coloneqq m - s$. Consider $\mathcal{G}_P = (V_P,E_P)$ to be the bipartite graph induced by the matrix $P$, where the set of vertices $[m] = [1,\dots,s,s+1,\dots,s+\ell]$ corresponds to the sets of columns and rows of $P$ and the edges correspond to the pairs of rows $i \in [s]$ and columns $j \in [\ell]$ such that $P_{ij} \neq 0$.

	 \begin{proposition}
	 	\label[proposition]{lemma: connectedComponents}
	 	The connected components of $\mathcal{G}_P$ (as a graph) are in correspondence with the connected components of $M(C)$.
	 \end{proposition}

	 \begin{proof}
Let $i \in [s]$ and $j \in [\ell]$.	If $P_{ij} \neq 0$, then the $i$-th and $(s+j)$-th columns of $C$ are connected in $M(C)$, by considering the circuit $\mathcal{C} = \{i' \in [s] : \: P_{i'j} \neq 0\} \cup \{s + j\}$. As the connectivity relation is transitive \cite[Proposition 4.1.2]{oxley}, the connected components of $G_P$ equal the connected components of $M(C)$.
\end{proof}

In particular, the connected components of $G_P$ do not depend on the choice of the row-reduced echelon form \eqref{eq: rref} and the matrix $C$ has a block structure determined by each of the connected components of $G_P$. Finally, we introduce the notion of iterated circuits, which has been used in \cite{esterov2010newton, cattani2022nonsplittingflagsiteratedcircuits}, and will play an essential role in our characterization of degeneracy.
	 \begin{definition}
	 	\label[definition]{def: iteratedCircuits}
	 	An \emph{iterated circuit} of $M(C)$ is a sequence $(\mathcal{C}_1,\dots,\mathcal{C}_k)$ of subsets $\mathcal{C}_i \subset [m]$ such that:
	 	\begin{enumerate}[label=(\roman*)]
	 		\item $|\mathcal{C}_i| \geq 2$ for each $i = 1,\dots,k$,
	 		\item $\mathcal{C}_{i+1}$ is a circuit of the contraction $M(X_i)$, where $X_i = \bigcup_{j = 1}^i\mathcal{C}_j$ for $i = 0,\dots,k-1$.
	 	\end{enumerate}
	 \end{definition}

	 \begin{remark}
	 	The fact that $\mathcal{C}_i \cap \mathcal{C}_{i'} = \varnothing$  for all $i \neq i'$ follows from the definition. Each of the sets $\mathcal{C}_i$ is contained in a unique connected component of $M(C)$, and several of them may lie in the same component. For each connected component $\mathcal{B}$ of $M(C)$, the subsequence of those $\mathcal{C}_i$ contained in $\mathcal{B}$ is an iterated circuit of $M(C)_{\mid \mathcal{B}}$.
	 \end{remark}

 The dual object of an iterated circuit in $M(C)$ is a flag of flats in the dual matroid.

     \begin{proposition}
    \label[proposition]{prop: dualityIteratedFlags}
  Let $C \in \mathbb{C}^{s \times m}$ be a matrix of rank $s$ and let
$D \in \mathbb{C}^{m \times \ell}$ be a Gale dual matrix.  The tuple
$(\mathcal{C}_1,\dots,\mathcal{C}_k)$ for $\mathcal{C}_i \subset [m]$ is an iterated circuit of $M(C)$ after removing the singleton elements if and only if the chain
$$E_{\ell - k} \subset \dots \subset E_{\ell} = [m], \quad
E_{i} \coloneqq [m] \setminus \bigcup_{j = 1}^{\ell - i}\mathcal{C}_j$$
is a flag of flats in $M(D)$ with $r_D(E_i) = i$ for all
$i = \ell - k,\dots,\ell$. Moreover,
$$r_C\big(\bigcup_{i = 1}^k\mathcal{C}_i\big) = s \iff
E_{\ell - k} \text{ is independent in } M(D).$$
     \end{proposition}

    \begin{proof}
    We know that $E$ is a flat of rank $\ell - 1$ in $M(D)$ (also called a hyperplane) if and only if $[m] \setminus E$ is a circuit in $M(C)$ \cite[Proposition 2.1.6(iii)]{oxley}. By \eqref{eq: contractionRestriction}, $E_{\ell - i}$ is a hyperplane in $M(D)_{\mid E_{\ell -i + 1}}$ (in particular, a flat in $M(D)$ of rank $\ell - i$) for $i = 1,\dots,k$ if and only if $\mathcal{C}_{i} = E_{\ell - i + 1} \setminus E_{\ell - i}$ is a circuit in $M(\bigcup_{j = 1}^{i-1} \mathcal{C}_j)$ for $i = 1,\dots,k$. Therefore, $E_{\ell - k} \subset \dots \subset E_{\ell}$ is a chain of flats in $M(D)$ if and only if $(\mathcal{C}_1,\dots,\mathcal{C}_k)$ is an iterated circuit of $M(C)$, up to removing the sets with one element. Using \eqref{rank:dual} the rank of $E_{\ell - k}$ in $M(D)$ equals
    $$r_D(E_{\ell - k}) = \big|E_{\ell - k}\big| + r_C\big(\bigcup_{i = 1}^{k} \mathcal{C}_{i}\big) - s.$$
    Therefore, $E_{\ell - k}$ is independent in $M(D)$ if and only if $r_C\big(\bigcup_{i = 1}^k \mathcal{C}_{i}\big) = s$.
    \end{proof}

 \subsection{Vertical systems and generic consistency}

We consider a vertical system as in \eqref{verticalSystem}, i.e. a parametric polynomial system of the form
	 \[ F =  C (u \star x^B) \in \CC[u,x^\pm]^s \]
	 with variables $x=(x_1,\dots,x_n)$ and parameters in the ring $\mathbb{C}[u_1,\dots,u_m]$. Throughout the article,
	 $B\in \ZZ^{n\times m}$ is a matrix whose columns we denote by $b_1,\dots,b_m$  and $C\in \CC^{s\times m}$ is a matrix of rank $s$, whose columns we denote by $c_1,\dots,c_m \in \mathbb{C}^s$.

     By $F_{u}$ we refer to the specialization of $F$ to a given $u \in \CC^m$ and we denote
     $$V_{\Tn}(F_u) \coloneqq \{x \in \Tn: \: F_u(x) = 0\}.$$

The partition of $M(C)$ can be related to the invariance of the polynomial system under an action of the torus \cite{feliu2025toricinvarianceverticallyparametrized}. In addition, the graph $G_P$ in \Cref{lemma: connectedComponents} has been used in \cite{checa2025effectivecriterionmultiplepositive} in the context of positive solutions of vertically parametrized polynomial systems, as well as in the context of Gale duality of point configurations \cite{rodriguezpajares2025combinatorialdescriptionselfassociatedset}. In particular, the following modifications can be performed in a given specialization $F_u$ of a vertical system $F = C(u \star x^B)$, without changing its geometry.

	 \begin{proposition}
	 	\label[proposition]{thm: connectedComponentsInvariant}
	 	Let $F = C(u \star x ^B)$ be a vertical system and let $\mathcal{B} \subset [m]$ be a connected component of $M(C)$.
	 	\begin{enumerate}[label=(\roman*)]    \item Let $u \in \mathbb{C}^m$ and $\lambda \in \mathbb{C}^*$. Consider $u' \in \mathbb{C}^m$ given by
	 		$$u'_i \coloneqq \begin{cases}
	 			\lambda u_i & i \in \mathcal{B} \\
	 			u_i & i \notin \mathcal{B}
	 		\end{cases},$$
	 		Then, $V_{\Tn}(F_{u'}) = V_{\Tn}(F_{u})$.
	 		\item Consider $a \in \mathbb{Z}^n$ and the vertical system $F' = C(u \star x^{B'})$ for some matrix $B' \in \mathbb{Z}^{n \times m}$ whose columns are
	 		$$b_i' = \begin{cases}
	 			b_i + a & i \in \mathcal{B} \\
	 			b_i  & i \notin \mathcal{B}
	 		\end{cases}.$$
	 		Then, $V_{\Tn}(F_{u}) = V_{\Tn}(F'_{u})$ for all $u \in \mathbb{C}^m$.
	 	\end{enumerate}
	 \end{proposition}

	 \begin{proof}
	 	From \Cref{lemma: connectedComponents}, we deduce that, if we write $C$ as in \eqref{eq: rref}, it must have a block structure where each block corresponds to a single connected component (this observation has been made in \cite[Remark 4.6]{feliu2025toricinvarianceverticallyparametrized}). Therefore, if $s' = r_C(\mathcal{B})$ we may consider the equations $F_1 = \dots = F_{s'} = 0$ that correspond to the block of $\mathcal{B}$.
	 	\begin{enumerate}[label=(\roman*)]
         \item The $u'$ chosen as above corresponds to scaling the polynomial equations in the connected component $\B$ by $\lambda$, which does not change the zero set over $\Tn$.

         \item The $b'$ chosen as above is equivalent to multiplying the equations in the connected component $\B$ by the monomial $x^a$ and does not change the zero locus over $\Tn$.  \qedhere
        \end{enumerate}
	 \end{proof}

In the following, we focus on the property of generic consistency of $F$, which implies that there exists an open subset where $F_u$ has the expected dimension.

   \begin{definition}
        Let $F = C(u \star x^B)$ be a vertical system.  With these notations in place, we say $F$ is \textit{generically consistent} if there exists a nonempty Zariski open subset $U \subset \mathbb{C}^{m}$ such that $\dim V_{\Tn}(F_{u}) = n - s$ for all $u \in U$.
     \end{definition}

In \cite{FELIU2025630}, the authors proved the following theorem and provided a verifiable condition, based on linear algebra, for determining generic consistency. Although other equivalent conditions were given, we summarize their criterion as in the following theorem.

	 \begin{theorem} \cite[Theorem 3.7]{FELIU2025630}
	 	\label{thm:genericConsistencyOSkarElisendaBeatriz}    Given a vertical system $F = C(u \star x^B)$, the following are equivalent:
	 	\begin{enumerate}[label=(\roman*)]
	 		\item The system $F$ is generically consistent.
	 		\item There exists a vector $w \in \ker(C)$ such that the matrix $ C\diag(w)B^{\top}$ has rank $s$.
	 	\end{enumerate}
	 \end{theorem}

In addition, they showed that if
$F$ is generically consistent, then there is an open set where its solutions are non-degenerate, meaning their Jacobian has full rank. If
$F$ is not generically consistent, then for any parameter values where $V_{\Tn}(F_u)$ is nonempty, the solutions are degenerate.

\begin{remark}
	\label[remark]{rk: reduce cases}

	If there exists a row of $C$ where only one entry is nonzero, one can easily rule out the possibility of generic consistency. Indeed, the equation corresponding to that row is of the form $u \cdot x^{b} = 0$ for some $b \in \mathbb{Z}^n$. For $u \neq 0$, we get $V_{\Tn}(F_{u}) = \varnothing$. This rules out the case where $\ell = 0$. Moreover, we also exclude the case when $C$ has a zero column, as such a column would not appear in the polynomial system $F = C(u \star x^B)$.
	Therefore, given a matrix $C \in \CC^{s \times m}$, we assume that $m > s$ and that
	$M(C)$ has no loops or coloops, i.e. $C$ has no zero columns and each row of \eqref{eq: rref} has at least two nonzero entries. All in all, up to a reordering of the columns in \eqref{eq: rref}, we may assume that $P \in \mathbb{C}^{s \times \ell}$ has no zero rows or columns.
\end{remark}

	 Note that the condition on the rank of $C\diag(w)B^{\top}$ can be reduced to checking whether a certain polynomial with coefficients in the vector space $\wedge^s \mathbb{C}^n$ is nonzero.

	 \begin{proposition}
	 	\label[proposition]{lemma_1}
	 	Let $F = C(u \star x^B)$ and assume that $C \in \mathbb{C}^{s \times m}$ is of the form \eqref{eq: rref} for some $P \in \mathbb{C}^{s \times \ell}$. There exists $w \in \ker(C)$ such that $C\diag(w)B^{\top}$ has rank $s$ if and only if
	 	\begin{equation}
	 		\label{eq:consistencyPolynomial}
	 		q(\lambda) = \sum_{\varphi: [s] \xrightarrow[]{} [\ell]}\bigwedge_{i = 1}^s \lambda_{\varphi(i)}  P_{i,\varphi(i)}\big(b_i - b_{s + \varphi(i)}\big)
	 	\end{equation}
	 	is a nonzero polynomial in $\big(\wedge^s \mathbb{C}^n\big)[\lambda_1,\dots,\lambda_{\ell}]$.

	 \end{proposition}
\begin{proof}
    Consider the matrix $D \in \mathbb{C}^{m \times \ell}$
    $$D = \begin{pmatrix}
        -P \\
        \Id_{\ell}
    \end{pmatrix},$$
    whose columns span the kernel of $C$. In order to check that condition $(ii)$ in \Cref{thm:genericConsistencyOSkarElisendaBeatriz} holds, we have to check that the matrix $C \diag(D \lambda) B^{\top}$ has rank $s$ for some $\lambda \in \mathbb{C}^{\ell}$. Expanding on $C \diag(D \lambda)$, we deduce
    $$C \diag(D\lambda) = \begin{pmatrix}
        -\sum_{j = 1}^{\ell}P_{1j}\lambda_j & 0 & \dots & 0 & P_{11}\lambda_1 & \dots & P_{1\ell}\lambda_{\ell} \\
        0 & -\sum_{j = 1}^{\ell}P_{2j}\lambda_j & \dots & 0 & P_{21}\lambda_1 & \dots & P_{2\ell}\lambda_{\ell}  \\
        \dots & & \dots & \dots & \dots && \dots \\
        0 & \dots & 0 &-\sum_{j = 1}^{\ell}P_{sj}\lambda_j &  P_{s1}\lambda_1 & \dots & P_{s\ell} \lambda_{\ell} \\
    \end{pmatrix}$$
    which implies that
    \begin{equation}
        \label{equatio_Consistencyn}
        C \diag(D\lambda)B^{\top} = \begin{pmatrix}
            \sum_{j = 1}^{\ell}P_{1j}\lambda_j\big(b_{s + j} - b_{1} \big) \\
            \sum_{j = 1}^{\ell}P_{2j}\lambda_j\big(b_{s + j} - b_{2} \big) \\
            \dots \\
            \sum_{j = 1}^{\ell}P_{sj}\lambda_j\big(b_{s + j} - b_{s} \big) \\
        \end{pmatrix}.
    \end{equation}
    This matrix has rank $s$ for some choice of $\lambda \in \mathbb{C}^{\ell}$ if and only if the polynomial in $\wedge^s \mathbb{C}^n \otimes \mathbb{C}[\lambda_1,\dots,\lambda_{\ell}]$ where the coefficient of $e_{i_1} \wedge \dots \wedge e_{i_s}$ is the minor of $C\diag(D\lambda)B^{\top}$ corresponding to the columns $\{i_1,\dots,i_s\}$ is nonzero. This polynomial can be written as the polynomial $q(\lambda)$ in \eqref{eq:consistencyPolynomial} by using that the determinant is multilinear and choosing the summand corresponding to $\varphi(i) \in [\ell]$ in the $i$-th row of the matrix \eqref{equatio_Consistencyn}.
\end{proof}

	 	\Cref{lemma_1} implies that verifying that the system $F = C(u \star x^B)$ is generically consistent can be reduced to verifying that one of the coefficients of the monomials in $q(\lambda)$ is nonzero. Although $q(\lambda)$ may depend upon the choice of the row-reduced echelon form, the condition $q(\lambda) \equiv 0$ cannot depend on this choice.

\subsection{Iterated circuits and Newton polytopes}
In this subsection, we consider the iterated circuits of $M(C)$ and the associated Newton polytopes obtained from the columns of $B$ for each circuit. As a main result, we will give an explicit condition that ensures degeneracy of the system. For each subset $X \subset [m]$,
we denote by $\Delta_{X}$ the Newton polytope generated by the columns of $B$ restricted to $X$, i.e.
$\Delta_{X} \coloneqq \conv \{ b_j : \: j \in X \} \subset \mathbb{R}^n.$
Additionally, we set $\Delta \coloneqq \Delta_{[m]}$. In some cases, \Cref{lemma_1} can already be used to deduce the generic consistency of $F = C(u \star x^B)$.

\begin{corollary}
\label[corollary]{corollary:circuits}
Let $F = C(u \star x^B)$ be a vertical system.
If there exists a circuit $\mathcal{C} \subset [m]$ of $M(C)$ such that $\dim \big(\Delta_{\mathcal{C}} \big) = s$, then $F$ is generically consistent.
\end{corollary}

\begin{proof}
    We can assume that the first $s + 1$ columns of $C$ correspond to $\mathcal{C}$ and deduce that $\prod_{i = 1}^{s}P_{i1} \neq 0$. With this, we deduce that the coefficient of $\lambda_1^s$ in $q(\lambda)$ equals
	$$\bigwedge_{i = 1}^s P_{i1}\big( b_{s + 1} - b_i \big),$$
	which is nonzero if and only if $\dim \big(\Delta_{\mathcal{C}} \big) = s$.
\end{proof}

 In particular, when $\ell = 1$, we have that $F = C(u \star x^B)$ is generically consistent if and only if $\dim \Delta = s$. We will generalize this aspect in \Cref{lemmacircuits}.

\begin{proposition}
	\label[proposition]{lemma:simplices}
	Let $F = C(u \star x^B)$ be a vertical system and let $(\mathcal{C}_1,\dots,\mathcal{C}_k)$ be an iterated circuit of the matroid $M(C)$ satisfying
	\begin{equation}
		\label{eq:maximalDimension}
		\dim\Big(\sum_{i = 1}^k\Delta_{\mathcal{C}_i}\Big) = s.\end{equation}
	Then,
	\begin{enumerate}[label=(\roman*)]
		\item $\dim(\Delta_{\mathcal{C}_i}) = |\mathcal{C}_i| - 1$ for all $i \in \{1,\dots,k\}$.
		\item $\sum_{i = 1}^k\dim\Delta_{\mathcal{C}_i} = s$.
	\end{enumerate}

\end{proposition}

\begin{proof}
	For each $i = 1,\dots,k$, $\dim \Delta_{\mathcal{C}_i} \leq |\mathcal{C}_i| - 1$ where equality holds if and only if the points $b_j$ for $j \in \mathcal{C}_i$ are affinely independent. Taking the sum for all $i \in \{1,\dots,k\}$, we get
	$$\dim(\sum_{i = 1}^k \Delta_{\mathcal{C}_i}) \leq \sum_{i = 1}^k\dim\Delta_{\mathcal{C}_i} \leq \big(\sum_{i = 1}^k|\mathcal{C}_i|\big) - k  = r_C\big(\bigcup_{i=1}^k\mathcal{C}_i\big)\leq s,$$
	where the equality follows from the definition of iterated circuits. Consequently, if $\dim\Big(\sum_{i = 1}^k\Delta_{\mathcal{C}_i}\Big) = s$, then each $\Delta_{\mathcal{C}_i}$ has dimension $|\mathcal{C}_i| - 1$ and $\sum_{i = 1}^k\dim\Delta_{\mathcal{C}_i} = s$.
\end{proof}

In the following example, we show that it is possible that none of the maximal circuits $\mathcal{C}$ of $M(C)$ (those with $|\mathcal{C}| = s + 1$) satisfy $\dim(\Delta_{\mathcal{C}}) = s$, but there exists an iterated circuit satisfying \eqref{eq:maximalDimension}.

\begin{example}
	\label[example]{example: pyramid}
	Consider $F = C(u \star x^B)$ to be the vertically parametrized polynomial system $F = C(u \star x^B)$ for the matrices:
	$$C = \begin{pmatrix}
		1 & 0 & 0 & 1 & 0 \\
		0 & 1 & 0 & 1 & 1 \\
		0 & 0 & 1 & 0 & 1
	\end{pmatrix}, \quad B = \begin{pmatrix}
		0 & 0 & 1 & 1 & 0 \\
		0 & 0 & 0 & 1 & 1\\
		0 & 1 & 0 & 0 & 0
	\end{pmatrix}$$
	which corresponds to the polynomial system
	\begin{equation}
		\label{second_system}
		\begin{cases}
			u_1 + u_4 xy = 0 ,\\[4pt]
			u_2z + u_4 xy + u_5 y = 0 ,\\[4pt]
			u_3x + u_5 y = 0 ,\\
		\end{cases}
	\end{equation}
	The only circuit $\mathcal{C}$ of $M(C)$ such that $|\mathcal{C}| = 4$ is the one corresponding to the columns $\{1,3,4,5\}$. However, the exponents corresponding to these columns are contained in the plane $\{z = 0\}$ and $\dim \Delta_{\mathcal{C}} = 2$. Nonetheless, if we consider the iterated circuit $(\mathcal{C}_1,\mathcal{C}_2)$ of $M(C)$ corresponding to the columns $\{1,2,4\}$ and $\{3,5\}$, then the polytopes $\Delta_{\mathcal{C}_1}$ and $\Delta_{\mathcal{C}_2}$ are
	$$\Delta_{\mathcal{C}_1} = \conv\{(0,0,0),(0,0,1),(1,1,0)\}, \quad \Delta_{\mathcal{C}_2} = \conv\{(1,0,0),(0,1,0)\}$$
	whose sum has dimension $3$.
	\end{example}

Thus, the previous example motivates the need for iterated circuits for studying generic consistency. We now state a general result on matroids, which gives a criterion for ruling out generic consistency.

\begin{theorem}
	\label[theorem]{lemmacircuits}
	Let $M$ be a matroid of rank $s$ on the index set $[m]$ and let $b_1,\dots,b_m \in \mathbb{R}^n$. Assume that for every circuit $\mathcal{C}$ of $M$ we have
	\begin{equation}
		\label{eq:hypotheselemmacircuits}
		\dim \Delta_{\mathcal{C}} < |\mathcal{C}|-1,
	\end{equation}
	then, for any function $\varphi: \{1,\dots,s\} \to \{s + 1,\dots,m\}$, the vector space $$L_{\varphi} = \spann_{\mathbb{R}}\{ b_i - b_{\varphi(i)} : \: i \in [s]\} \subset \mathbb{R}^n$$ satisfies  $
	\dim_{\RR} L_{\varphi} < s.$
\end{theorem}

\begin{proof}
	Let $X \subseteq [m]$ be a subset with $|X| \ge s+1$, and let $d = \dim \Delta_{X}$. Choose a maximal subset $Y\subseteq X$ such that the set $\{b_i\}_{i \in Y}$ is affinely independent. The number of elements in $Y$ must be $|Y| = d+1$. Since $M$ has rank $s$, any subset of size strictly bigger than $s$ is dependent.
	If $d \ge s$, then $|Y| = d+1 \ge s+1$. Equivalently, there exists a circuit $\mathcal{C} \subset Y$. The set $\{b_i\}_{i \in \mathcal{C}}$ corresponding to  $\mathcal{C}$ is affinely independent, implying that $\dim \Delta_{\mathcal{C}} = |\mathcal{C}|-1$, contradicting \eqref{eq:hypotheselemmacircuits}. Therefore $d < s$.

	Fix a function $\varphi: [s] \xrightarrow{} \{s+1,\dots,m\}$ and consider the set $X = [s] \cup \varphi([s])$.
	Since $\varphi([s]) \subseteq \{s+1,\dots,m\}$ is nonempty, we have $|X| \ge s+1$, implying that $\dim \Delta_{X} < s$. Moreover, all the vectors in $L_{\varphi}$ are contained in the linear span of $\Delta_X$ (see \eqref{eq: directionSpaces} later on). Therefore, $\dim L_{\varphi} \le \dim \Delta_{X} < s$.
	\qedhere
\end{proof}

\begin{corollary}
	\label[corollary]{corollary:circuit}
    Let $F = C(u \star x^B)$ be a vertical system. Assume that for all circuits  $\mathcal{C} \subset [m]$ of $M(C)$, we have $\dim(\Delta_{\mathcal{C}}) < |\mathcal{C}| - 1.$ Then, the polynomial $q$ in \eqref{eq:consistencyPolynomial} is identically zero. In particular, $F$ is not generically consistent.
\end{corollary}

\begin{proof}
	This follows from \Cref{lemmacircuits}, applied to the terms of the polynomial $q(\lambda)$ in \eqref{eq:consistencyPolynomial}. The second part follows from \Cref{lemma_1}.
\end{proof}

\subsection{Gale duality}

Gale duality has been established as a central tool in enumerative geometry of sparse systems, for studying complex, real, or positive solutions \cite{sottile2011real, bihan2021optimaldescartes}. In this section, we show that the structure of vertical systems is adequate to introduce Gale duality in parametric polynomial systems.

Let $F = C(u \star x^B)$ be a vertical system with $C \in \mathbb{C}^{s \times m}$ and $B \in \mathbb{Z}^{n \times m}$. We let $\overline{B} \in \mathbb{Z}^{(n + 1) \times m}$ be the matrix $B$ after adding a first row equal to $\mathbf{1} \in \mathbb{Z}^m$, i.e.
$$\overline{B} = \begin{pmatrix}
    1 & 1 & \dots & 1 \\
    b_1 & b_2 & \dots & b_m
\end{pmatrix}.$$
Let $D \in \mathbb{C}^{m \times \ell}$ be a Gale dual matrix of $C$ with rows $d_1,\dots,d_m \in \mathbb{C}^{\ell}$, i.e. its columns generate the kernel of $C$. We let $G \in \mathbb{Z}^{m \times (m - 1 - n)}$ be a Gale dual matrix of $\overline{B}$, whose columns generate the kernel of $\overline{B}$ as an abelian group. Finally, we let $g_1,\dots,g_m \in \mathbb{Z}^{m - n -1}$ be the rows of $G$. Here, we assume that $\overline{B}$ has rank $n+1$.

The \textit{Gale dual system} of $F$ is the system
\begin{equation}
\label{eq:galeDualSystem}
H \coloneqq
\left(D(y)\star\frac{1}{u}\right)^G-\mathbf 1_{m-n-1} \in \mathbb{C}[u^{\pm},y_1^{\pm},\dots,y_{\ell}^{\pm}],
\end{equation}
where $D(y)\in\mathbb C[y]^m$ denotes the vector whose $i$-th
entry is the linear form $\langle d_i,y\rangle$.

\begin{remark}
\label[remark]{rk: equationsOfDual}
We note that, as the columns of $G$ generate the kernel over $\mathbb{Z}$, the greatest common divisor of the entries of the columns of $G$ is 1.  Secondly, for
$j\in\{1,\dots,m-n-1\}$, the $j$-th equation of
\eqref{eq:galeDualSystem} is
\[
\prod_{i=1}^m
\left(\frac{\langle d_i,y\rangle}{u_i}\right)^{g_{ij}}=1.
\]
As $\mathbf{1}_m \in \mathbb{Z}^m$ is a row of $\overline{B}$, the left-hand side of this identity is a rational function of degree $0$. Thus, we study the zero locus of this equation in $\mathbb{P}^{\ell - 1}$ when each of the factors is a nonzero complex number, i.e. $u_i\langle d_i, y\rangle \neq 0$. After clearing denominators, we recover the polynomial equation
\[
\prod_{g_{ij}>0}\Big(\frac{\langle d_i,y\rangle}{u_i}\Big)^{g_{ij}}
-
\prod_{g_{ij}<0}\Big(\frac{\langle d_i,y\rangle}{u_i}\Big)^{-g_{ij}}
=0.
\]

\end{remark}

Note that if $u_i = 0$, we may consider the Gale dual of the system $F$ after removing the $i$-th column in $C$. Thus, for $u \in (\mathbb{C}^*)^m$, we consider $H_u$ to be the specialization of \eqref{eq:galeDualSystem} and the zero set:
\begin{equation}
\label{eq: zeroSetVD}
    V_D(H_u) = \{y \in \mathbb{P}^{\ell - 1} :\: H_u = 0, \: \langle d_i,y \rangle \neq 0, \: \forall i \in \{1,\dots,m\} \}
\end{equation}

\begin{theorem}
\label[theorem]{thm:galeDualityVerticalSystems}
Let \(F=C(u\star x^B)\) be a vertical system, and let
\[
H=\left(D(y)\star\frac{1}{u}\right)^G-\mathbf{1}_{m-n-1}
\]
be its Gale dual. Then, for every \(u\in(\mathbb{C}^*)^m\),
\begin{enumerate}[label=(\roman*)]
    \item If the columns of $\overline B$ generate $\mathbb{Q}^{n+1}$, \(\dim V_{\Tn}(F_u)=\dim V_D(H_u)\).
    \item If the columns of $\overline B$ generate $\mathbb{Z}^{n+1}$,
then Gale duality induces an isomorphism of schemes
$
V_{\Tn}(F_u)\cong V_D(H_u).
$
\end{enumerate}
\end{theorem}

One can find a proof of this statement in \cite[Theorem~6.3]{sottile2011real}, assuming that
the columns of \(\overline{B}\)
generate \(\mathbb{Z}^{n + 1}\), or equivalently, that $B$ contains the zero column and the rest generate $\mathbb{Z}^n$ (up to a translation of $B$). Without this assumption, one may write the matrix
$\overline{B}=Q\widetilde B$, where $Q$ is a square nonsingular matrix and the columns of
\(\widetilde B\) generate \(\mathbb{Z}^{n + 1}\). The monomial
map induced by $Q$ is finite and surjective, so it preserves the dimension of the
solution set, although not necessarily its scheme structure or number
of connected components. Moreover, \(\overline B\) and \(\widetilde B\) have the
same Gale dual, proving the dimension equality in (i).

\subsection{Projected systems}
	Given $X\subset [m]$, we define the real vector space:
	\begin{equation}
    \label{eq: directionSpaces}
    L_X \coloneqq \spann_{\RR}(b_i - b_j : \: i,j \in X) \subset \mathbb{R}^n.\end{equation}
	We denote by  $L_{X}^{\perp} \subset \mathbb{R}^n$ the vector space orthogonal to $L_X$
	and $\tau_X: \mathbb{R}^n \xrightarrow{} L_X^{\perp}$ denotes a general linear projection with kernel $L_X$ and $\tau_X(\mathbb{Z}^n) \subset \mathbb{Z}^n$. For each subset $X \subset [m]$, we denote by $B_X$ the matrix with columns $(\tau_X(b_j))_{j \notin X}$. With this notation, we define the corresponding projected system as follows.

\begin{definition}
Let $F=C(u\star x^B)$ be a vertical system, where $C\in\mathbb{C}^{s\times m}$ and $B\in\ZZ^{n\times m}$. Let $X\subseteq [m]$ and let $C_X \in \CC^{s' \times m'}$ be the matrix in \eqref{equation:contraction}. The \textit{projected system} with respect to $X$ is
	\begin{equation}
		\label{eq: systemFI}
		F_{X} = C_{X}(u \star x^{B_{X}}).
	\end{equation}
We write $\ell_{X} \coloneqq m' - s'$.
\end{definition}

In \Cref{prop: contractedCircuit}, we show that if a vertical system is generically consistent, then there exists a circuit $\mathcal{C}$ of $M(C)$ such that the projected system $F_{\mathcal{C}}$ is also generically consistent. The proof of \Cref{prop: contractedCircuit} is based on the following lemma, which has a direct analogue in the theory of hyperplane arrangements and zonotopal algebras (\cite[Theorem~5.8]{HoltzRon2011} and \cite[Proposition~4.15 and Theorem~4.17]{ArdilaPostnikov2010}).

Given a vector space $V \subset \mathbb{C}^m$ and $c \in V$, we denote by $\supp(c) = \{i \in [m] : \: c_i \neq 0\} \subset [m]$ its support. We say that a vector $c \in V$ is a \textit{circuit vector} if for every other $c' \in V \setminus \{0\}$ with $\supp(c') \neq \supp(c)$, we have $\supp(c') \not\subset \supp(c)$. For any homogeneous polynomial function $P: V \xrightarrow[]{} \CC$ and $c \in V$, we denote the derivative in the direction of $c$ as:
$$\partial_c^iP(w) \coloneqq \frac{d^i}{d\lambda^i}P(w + \lambda c)_{\mid \lambda = 0}.$$

\begin{lemma}
\label{lemma: derivativeCircuit}
    Let $V \subset \mathbb{C}^m$ be a vector space of dimension $\ell$ that is not contained in any coordinate hyperplane and let $P: V \xrightarrow[]{} \CC$ be a homogeneous polynomial function of degree greater or equal than $m - \ell$. If for all circuit vectors $c$ with $k = |\supp(c)|$, we have that
    $$\partial_c^{k-1}P = 0,$$
    then $P = 0$.
\end{lemma}

\begin{proof}
The assertion for a polynomial of degree $d > m - \ell$ can be reduced to the case of a polynomial of degree $m - \ell$ by taking a nonzero partial derivatives of degree $d - m + \ell$.

For the case where $P$ has degree $m - \ell$, the proof follows by induction on $\ell$. If $\ell = 1$, then the only circuit vectors in $V$ have support equal to $[m]$ (as $V$ is not contained in any coordinate hyperplane). If we have a polynomial of degree $m - 1$ whose $(m - 1)$-th derivative equals $0$, then it must be the zero polynomial.

If $\ell > 1$, we may consider the coordinate functions $L_j: V \xrightarrow[]{} \CC$ sending a vector $w$ to $w_j$ and the vector spaces $T_j = \ker(L_j)$. Clearly, $T_j \subset V$ is a hyperplane as $V$ is not contained in any coordinate hyperplane. Moreover, any circuit vector in $T_j$ is also a circuit vector in $V$. For each $j \in \{1,\dots,m\}$, consider
$$\mathcal{X}_j = \{j' \in [m] :\: \ker L_{j'} = T_j\}$$
Note that $[m] = \bigsqcup_{j \in \mathcal{X}}\mathcal{X}_{j}$ for some set of representatives $\mathcal{X} \subset [m]$ and let $a_j = |\mathcal{X}_j|$ for $j \in [m]$. In particular, we can always restrict the support of $T_j$ to $\C^{m - a_j}$. For any $v \in V \setminus T_j$ and $0 \leq t \leq a_j - 1$, we may consider the function
$$P_t \coloneqq \partial_v^tP_{\: \mid T_j},$$
which has degree
$m - \ell - t \geq m - a_j - \ell + 1$. Using the hypothesis and that derivatives and restrictions commute, we derive that the functions $P_t$ satisfy the conditions of the lemma. More concretely, for all circuit vectors $c \in T_j$, we have
$$\partial_c^{k - 1}P_t=\partial_c^{k - 1}\partial_v^{t}P_{\mid T_j} = 0.$$
Therefore, $P_t = 0$ in $T_j$ for all $0 \leq t \leq a_j - 1$. As the function $P(w + \lambda v)$ equals
$\sum_{t \geq 0}\frac{\lambda^t}{t!}
(\partial_v^tP)(w)$, the polynomial $L_j^{a_j}$ divides $P$. If we apply this for all $j \in \mathcal{X}$, we deduce that $\prod_{j \in \mathcal{X}} L_j^{a_j}$ (which has degree $m > m - \ell$) divides $P$, implying that $P$ must be the zero polynomial in $V$.
\end{proof}

The following proposition follows from applying the previous lemma to the case where $V = \ker(C)$ and $P$ is a minor of the matrix $C \diag(w) B^{\top}$.

\begin{proposition}
\label{prop: contractedCircuit}
    Let $F = C(u \star x^B)$ be a vertical system. If $F$ is generically consistent, there exists a circuit $\mathcal{C}$ of $M(C)$ with $\dim \Delta_{\mathcal{C}} = r_C(\mathcal{C})$ such that $F_{\mathcal{C}}$ is generically consistent.
\end{proposition}
\begin{proof}
    As $F$ is generically consistent, \Cref{thm:genericConsistencyOSkarElisendaBeatriz} implies that there exists a minor of the matrix $C\diag(w)B^{\top}$ which is a nonzero polynomial function on $\ker C \subset \mathbb{C}^m$.

    We denote this minor as $P: \ker(C) \xrightarrow{} \mathbb{C}$ and note that it is a function of degree $s = m - \ell$. Therefore, \Cref{lemma: derivativeCircuit} implies that there exists a circuit vector $c \in \ker C$, i.e. a circuit $\mathcal{C}$ of $M(C)$, such that $\mC = \supp(c)$ and $r = r_C(\mathcal{C}) = |\mathcal{C}| - 1$ and $\partial_c^rP \neq 0$.

    Note that the matrix $C \diag(c)B^{\top}$ can have, at most, rank $r$, as it only involves the columns of $C$ indexed by $\mC$. Therefore, using the multilinearity of the chosen minor, the function $P(w + tc)$ has degree at most $r$ in $t$. Since $\partial^r_cP \neq 0$, there exists a vector $w \in \ker(C)$ such that the coefficient of $t^r$ in $P(w+tc)$ is nonzero. Next, we write $C$ in row-reduced echelon form as in \eqref{eq: rref} (in particular, preserving the function $P$ up to constants). This provides a matrix $Q$ for the contracted system $F_{\mathcal{C}}$, which must have rank $s - r$.

    As $c \in \ker C$, the matrix $R \coloneqq C\diag(c)B^{\top}$ has a first block whose rows span $L_{\mC}$ and then a block of zeros. On the other hand, if we let $Q \coloneqq C_{\mC}$, the bottom block of $C\diag(w + tc)B^{\top}$ has the form
    $$M = Q\diag(w')\widetilde{B}^{\top},$$
    where $w' = (w_j)_{j \notin \mathcal{C}}$ and $\widetilde{B}$ only preserves the columns of $B$ corresponding to $j \notin \mC$. All in all, the coefficient of $t^r$ in the function $P(w + tc)$ is a minor of the matrix
    \begin{equation}
    \label{eq: RM}
        \begin{pmatrix}
        R \\
        M
    \end{pmatrix}.
    \end{equation}
    As this coefficient is nonzero, the matrix in \eqref{eq: RM} must have rank $s$. In particular, $R$ has rank $r$ implying that $\dim \Delta_{\mC} = \dim L_{\mathcal{C}} = r$ and the rows in $M$ are linearly independent modulo the rows of $R$, which after projecting to $L_{\mC}^{\perp}$, implies that
    $$\rank(Q\diag(w')B_{\mC}^{\top}) = s - r.$$
    By \Cref{thm:genericConsistencyOSkarElisendaBeatriz}, this implies that $F_{\mC}$ is generically consistent.
\end{proof}

\section{Generic consistency, iterated circuits and $G$-spanning flags}
\label{section:3}

In this section, we provide the main characterization of generic consistency in terms of the combinatorial properties of the iterated circuits of $M(C)$ and their Newton polytopes.

\subsection{Iterated circuits and generic consistency}

The first result of this combinatorial characterization is the following theorem.

\begin{theorem}
	\label[theorem]{thm:mainTheorem}
	Let $C \in \mathbb{C}^{s \times m}$ be of rank $s$ and let $B \in \mathbb{Z}^{n \times m}$. The following are equivalent:
	\begin{enumerate}[label=(\roman*)]
		\item The vertically parametrized system $F = C(u \star x^B)$ is generically consistent.
		\item There exists an iterated circuit $(\mathcal{C}_1,\dots,\mathcal{C}_k)$ of $M(C)$,  such that
		$$\dim\big(\sum_{j = 1}^k\Delta_{\mathcal{C}_j}\big) = s.$$
	\end{enumerate}
\end{theorem}

\begin{proof}
	To prove that (ii) implies (i), we assume that there exists an iterated circuit $(\mathcal{C}_1,\dots,\mathcal{C}_k)$ such that $
		\dim\Big(\sum_{i = 1}^k\Delta_{\mathcal{C}_i}\Big) = s.$
	Let $d_i \coloneqq \dim(\Delta_{\mathcal{C}_i})$ for $i = 1,\dots,k$. By \Cref{lemma:simplices}, the set of points $\{b_j\}_{j \in \mathcal{C}_i}$ is affinely independent for $i \in \{1,\dots,k\}$. Therefore, we can pick $j_i \in \mathcal{C}_i$ such that
	$$\bigwedge_{i = 1}^k\bigwedge_{j \in \mathcal{C}_i \setminus \{j_i\}}(b_j - b_{j_i}) \neq 0.$$
	We note that $\mathcal{D} \coloneqq \cup_{i = 1}^k\mathcal{C}_i \setminus \{j_1,\dots,j_k\}$ is a basis of $M(C)$  and write $C$ in row-reduced echelon form with respect to $\mathcal{D}$. Let $P \in \mathbb{C}^{s \times \ell}$ be the matrix in \eqref{eq: rref} with respect to the basis $\mathcal{D}$, partitioning $[m]$ into $\mathcal{D} \sqcup ([m] \setminus \mathcal{D})$ . After reordering the columns, we may assume that $\{j_1,\dots,j_k\} = \{s + 1,\dots,s + k\}$, i.e. the first $k$ columns of $P$ correspond to the subset of the remaining elements in the iterated circuit. We denote $D_i = d_1 + \dots + d_i$ for all $i = 1,\dots,k$. As $\mathcal{C}_i$ is a circuit of $M(\mathcal{C}_1 \cup \dots \cup \mathcal{C}_{i-1})$, the support of the $i$-th column is contained in the first $D_i$ rows. In particular, $$P_{j,i} = \begin{cases}
		0 & D_i < j \\
		\neq 0 & D_{i-1} < j \leq D_i
	\end{cases}.$$ Using this, we deduce that the coefficient of the monomial $\prod_{i = 1}^k\lambda_i^{d_i}$ in the polynomial $q(\lambda)$ in \eqref{eq:consistencyPolynomial}
	equals:
	$$\prod_{i = 1}^k\prod_{j = D_{i-1} + 1}^{D_{i}}P_{j,i} \otimes \bigwedge_{i = 1}^k\bigwedge_{j = D_{i-1} + 1}^{D_i}(b_j - b_{s + i}) \neq 0.$$
 Here, we used the convention that $D_0 = 0$. Therefore, by \Cref{lemma_1}, $F = C(u \star x^B)$ is generically consistent.

	For the other implication, we assume that $F$ is generically consistent and proceed by induction on $\ell$. If $\ell = 1$, the polynomial $q(\lambda)$ must be of the form
	$$\Big(\prod_{i = 1}^sP_{i,1} \otimes \bigwedge_{i = 1}^s  \big(b_{s + 1} - b_i \big) \Big)\lambda_1.$$
	Therefore, $q(\lambda)$ is not a zero polynomial if and only if the unique circuit $\mathcal{C}$ in $M(C)$ satisfies $\dim \Delta_{\mathcal{C}} = s$.

    If $\ell > 1$, we use \Cref{prop: contractedCircuit} to deduce that there exists a circuit $\mathcal{C}_1$ such that $F_{\mathcal{C}_1}$ is generically consistent. If $d \coloneqq  r_C(\mathcal{C}_1)$, then the number of rows of $F_{\mC_1}$ is $s - d$, while the number of nonzero columns is, at most, $m - d - 1$. Therefore,
    $$\ell_{\mC_1} = m - d - 1 - (s - d) < \ell.$$
    Applying the induction hypothesis, we deduce that there exists an iterated circuit $\mathcal{C} = (\mathcal{C}_2,\dots,\mathcal{C}_k)$ of $M(\mathcal{C}_1)$ such that
    $$\dim\big(\sum_{j = 2}^k\tau_{\mathcal{C}_1}(\Delta_{\mathcal{C}_j})\big) = s - d$$
    As $\ker \tau_{\mathcal C_1} = L_{\mathcal{C}_1}$, $(\mC_1,\dots,\mC_k)$ is an iterated circuit of $M(C)$ with $\dim \big(\sum_{j = 1}^k\Delta_{\mC_j}\big) = s$, concluding the proof.
\end{proof}

\begin{example}
	\label[example]{example:section3}
	Consider the vertically parametrized system $F = C(u \star x^B)$ from \Cref{example: pyramid}. The set $\mathcal{C} = \{1,2,4\}$ is a circuit of $M(C)$. Additionally, $L_{\mathcal{C}} = \{z \in \mathbb{R}^3 : \: z_1 = z_2\}$ and its orthogonal space is $ L_{\mathcal{C}}^{\perp} = \{z \in \mathbb{R}^3 : \: z_3 = z_1 + z_2 = 0\}.$ Note that $L_{\mathcal{C}}^{\perp}$ is $1$-dimensional. We consider the linear projection
	$$\tau_{\mathcal{C}}: \mathbb{R}^3 \xrightarrow[]{} L_{\mathcal{C}}^{\perp}, \quad (z_1,z_2,z_3) \xrightarrow[]{} (z_1 - z_2, z_2 - z_1, 0).$$
	The contraction with respect to $I = \{1,2\}$ gives rise to the matrix
	$C_{\mathcal{C}} = \begin{pmatrix}
		1 & 1
	\end{pmatrix}$ which corresponds to the columns $\{3,5\}$ of $C$.
	Hence, the corresponding projected system $F_{\mathcal{C}} = C_{\mC}(u \star x^{B_{\mathcal{C}}})$ equals
	$$
	u_3xy^{-1} + u_5 x^{-1}y = 0.
	$$
	Although the polynomial defining $F_{\mC}$ involves two different variables, its Newton polytope lies in a $1$-dimensional linear space. We note that $\mathcal{C} = (\mathcal{C}_1,\mC_2)$ with $\mC_1 = \{1,2,4\}$ and $\mC_2 = \{3,5\}$ gives an iterated circuit of $M(C)$ such that $\dim (\Delta_{\mC_1}+\Delta_{\mC_2})=3$. By \Cref{thm:mainTheorem}, $F$ is generically consistent.
\end{example}

\subsection{$G$-spanning flags and generic consistency}

In this subsection, we give yet another characterization of generic consistency, this time using the Gale dual system of $F = C(u \star x^B)$. In particular, the characterization relies on the following definition.

\begin{definition}
\label[definition]{def: nonsplitting_flag}
Let $H = \left(D(y)\star\frac{1}{u}\right)^G-\mathbf{1}_{m-n-1}$ be the Gale dual of a vertical system $F = C(u \star x^B)$. A $G$-\textit{spanning flag} of length $k \in \mathbb{Z}_{> 0}$ is a flag of flats in $M(D)$ of the form
$$\emptyset = E_0 \subset E_1 \subset \dots \subset E_{k} $$
satisfying $r_D(E_j) = j$ for all $j = 1,\dots,k$ and if we define
\begin{equation}
\label{eq: si}
s_i \coloneqq \sum_{j \in E_i \setminus E_{i-1}}g_j, \quad i = 1,\dots,k,
\end{equation}
then,
\begin{equation}
\label{eq: Gspanning}
\spann_{\mathbb{R}} \langle s_j : \; j = 1,\dots,k - 1 \rangle = \spann_{\R} \langle g_j : \: j \in [m] \rangle.
\end{equation}
\end{definition}

\begin{remark}
\label[remark]{rk: excludesl}
The reason for excluding $s_k$ in \eqref{eq: Gspanning} is that if the flag is complete and $E_{\ell} = [m]$, then
$$\sum_{i = 1}^{\ell}s_i = \sum_{j = 1}^mg_j = 0,$$
implying that $s_{\ell}$ lies in the vector space spanned by $s_1,\dots,s_{\ell - 1}$.
\end{remark}

\begin{theorem}
\label{thm: GspanningFlag}
Let $F = C(u \star x^B)$ be a vertical system and let $H = \left(D(y)\star\frac{1}{u}\right)^G-\mathbf{1}_{m-n-1}$ be its dual vertical system. The following are equivalent:
\begin{enumerate}[label=(\roman*)]
    \item There exists a $G$-spanning flag of length $\ell$.
    \item There exists an iterated circuit $(\mathcal{C}_1,\dots,\mathcal{C}_k)$ of $M(C)$,  such that
            \begin{equation}
            \label{eq: sumDELTAS2}
				\dim\Big(\sum_{i = 1}^k\Delta_{\mathcal{C}_i}\Big) = s.
            \end{equation}
\end{enumerate}

\end{theorem}

\begin{proof}
Let $(\mathcal{C}_1,\dots,\mathcal{C}_k)$ be an iterated circuit. By \Cref{prop: dualityIteratedFlags}, there exists a flag
$$E_{\ell - k} \subset \dots \subset E_{\ell} = [m], \quad r_D(E_i) = i,$$
for $i = \ell - k,\dots,\ell$. Moreover, $E_{\ell - k}$ is independent if and only if $r_C(\cup_{i = 1}^k\mathcal{C}_i) = s$. If $E_{\ell - k}$ is independent, then the flag can be extended to a complete flag
\begin{equation}
\label{eq: completeFlag}
\emptyset = E_{0} \subset E_1 \subset \dots \subset E_{\ell} = [m], \quad r_D(E_i) = i,
\end{equation}
by considering any flag of independent sets contained in $E_{\ell - k}$.

Next, we consider the vector subspace of $\mathbb{R}^m$:
$$U \coloneqq \{x \in \RR^m : \: \supp(x) \subset \bigcup_{i = 1}^k\mathcal{C}_i, \: \sum_{j \in \mathcal{C}_i}x_j = 0, \: \forall i \in \{1,\dots,k\}\},$$
which has dimension $r_C\big(\bigcup_{i = 1}^k\mathcal{C}_i\big)$. In addition, consider the linear map defined by the matrix $\overline B$, i.e.
$$\overline{B}: \RR^{m} \xrightarrow[]{} \mathbb{R}^{n + 1}, \quad x \xrightarrow[]{} \overline{B} \cdot x.$$
Note that $U$ is the direct sum of the vector spaces
\[
\Big\{(x_j)_{j\in\mathcal C_i}\in\R^{|\mathcal C_i|} \: :
\sum_{j\in\mathcal C_i}x_j=0\Big\}.
\]
The image of each of these vector subspaces under $\overline{B}$ is isomorphic to the real vector spaces $L_{\mathcal{C}_i}$ from \eqref{eq: directionSpaces}, adding a first component equal to zero as the elements in $U$ satisfy that $\sum_{i = 1}^m x_i = 0$. Hence, we deduce that $$\dim \overline{B}(U) = \dim \big(\sum_{i = 1}^k\Delta_{\mathcal{C}_i}\big).$$
Using the rank-nullity theorem, we deduce that
\begin{equation}
\label{eq: rankNullity}
\dim\big(\sum_{i = 1}^k\Delta_{\mathcal{C}_i}\big) = r_C\Big(\bigcup_{i = 1}^k\mathcal{C}_i\Big) - \dim(U \cap \ker(\overline{B})).
\end{equation}
As the rank of $M(C)$ equals $s$, we have that $\dim(\sum_{i = 1}^k\Delta_{\mathcal{C}_i}) = s$ if and only if $r_C\big(\bigcup_{i = 1}^k\mathcal{C}_i\big) = s$ and $U \cap \ker(\overline{B}) = 0$.

Finally, we note that $x \in \ker(\overline{B})$ if and only if there exists $y \in \mathbb{R}^{m - n - 1}$ such that $x = Gy$, i.e. $x_i = \langle g_i, y\rangle$ for all $i = 1,\dots,m$. Using the complete flag in \eqref{eq: completeFlag}, we let $s_i$ be as in \eqref{eq: si} and note that $x \in U \cap \ker(\overline{B})$ if and only if
$$ \langle s_i,y \rangle = 0, \quad \forall i \in \{1,\dots,\ell\}, $$
for $x = G y$. Therefore, $U \cap \ker(\overline{B}) = 0$ which happens if and only if,
$$\spann_{\mathbb{R}} \langle s_j : \; j = 1,\dots,\ell \rangle = \spann_{\RR} \langle g_j : \: j \in [m] \rangle. $$
Using \Cref{rk: excludesl},  we conclude the proof of (ii)$\Rightarrow$(i). Conversely, by \Cref{prop: dualityIteratedFlags}, a complete flag of flats of $M(D)$ comes from an iterated circuit after omitting the sets $E_i\setminus E_{i-1}$ with one element, which impose $x_j=0$ on $U$ in both cases, and the same computation proves (i)$\Rightarrow$(ii).
\end{proof}

This result concludes the proof of \Cref{thm:mainIntro}. The following corollary is a natural consequence of this result.

\begin{corollary}
\label[corollary]{cor: GaleDual}
Let $H = (D(y) \star \frac{1}{u})^G - \mathbf{1}$ be a dual vertical system. The following are equivalent:
\begin{enumerate}[label=(\roman*)]
    \item There exists a $G$-spanning flag of length $\ell$.
   \item The dual vertical system $H = (D(y) \star \frac{1}{u})^G - \mathbf{1}$ is generically consistent.
\end{enumerate}
\end{corollary}

\begin{proof}
 This follows by combining \Cref{thm:galeDualityVerticalSystems}, \Cref{thm:mainTheorem} and \Cref{thm: GspanningFlag}.
 \end{proof}

 Under some conditions, the notion of $G$-spanning coincides with the non-splitting flags arising in \cite{cattani2022nonsplittingflagsiteratedcircuits, CURRAN2007115}. For the following definition, we consider the linear spaces
 $$L_{E}^* \coloneqq \spann_{\RR}\langle g_j \: : j \in E \rangle \subset \RR^{m - n - 1},$$
 where $E \subset [m]$ is a flat of $M(D)$.

 \begin{definition}
 \label{def: nonsplitting}
Let $F = C(u \star x^B)$ be a vertical system and let $H = (D(y) \star \frac{1}{u})^{G} - \mathbf{1}$ be its Gale dual system. We say that a flag of flats in $M(D)$
$$\emptyset \subset E_1 \subset \dots \subset E_{k}$$
is \emph{non-splitting} if
$s_i \notin L_{E_{i-1}}^*$ for all $i \in \{1,\dots,k-1\}$, where $s_i$ is defined in \eqref{eq: si}.
 \end{definition}

 \begin{corollary}
 \label[corollary]{cor: nonsplitGspan}
Let $F = C(u \star x^B)$ and let $H = (D(y) \star \frac{1}{u})^{G} - \mathbf{1}$ be its Gale dual. Assume that $s\leq n$. Let $0 \subset E_1 \subset \dots \subset E_\ell = [m]$ be a complete flag of flats in $M(D)$ with $r_D(E_i) = i$ with $i = 1,\dots,\ell - 1$.
\begin{enumerate}[label=(\roman*)]
    \item If the flag is non-splitting, then it is $G$-spanning.
    \item If the flag is $G$-spanning and $
    \dim L_{E_i}^* = i$ for all $i \in \{1,\dots,\ell - 1\}$, then it is non-splitting.
\end{enumerate}
\end{corollary}
\begin{proof}
Since $s\leq n$, we have
that the rank of $G$ equals $m-n-1\leq m-s-1=\ell-1$. Suppose first that the flag is non-splitting. Since
$$
\spann_{\R}\langle s_1,\dots,s_{i-1}\rangle
\subseteq L_{E_{i-1}}^*,
$$
the condition $s_i\notin L_{E_{i-1}}^*$ for
$1\leq i\leq\ell-1$ implies that
$s_1,\dots,s_{\ell-1}$ are linearly independent.
Together with the rank bound and the fact that $\sum_{i = 1}^{\ell} s_i = 0$, this gives
$$
\spann_{\RR}\langle s_1,\dots,s_{\ell}\rangle
=
\spann_{\RR}\langle g_j:j\in[m]\rangle.
$$
Thus the flag is $G$-spanning.

Conversely, suppose that the flag is $G$-spanning and
$\dim_{\RR}L_{E_i}^*=i$ for $1\leq i\leq\ell-1$.
Then
$$
L_{E_{\ell-1}}^*
\subseteq \spann_{\RR}\langle g_j:j\in[m]\rangle
=
\spann_{\RR}\langle s_1,\dots,s_{\ell-1}\rangle.
$$
Since the left-hand space has dimension $\ell-1$,
the vectors $s_1,\dots,s_{\ell-1}$ are linearly independent.
For each $1\leq i\leq\ell-1$, we therefore have
$$
\spann_{\RR}\langle s_1,\dots,s_i\rangle=L_{E_i}^*,
$$
because the left-hand space is contained in the right-hand
space and both have dimension $i$. Consequently,
$$
s_i\notin
\spann_{\RR}\langle s_1,\dots,s_{i-1}\rangle
=L_{E_{i-1}}^*,
\qquad 1\leq i\leq\ell-1.
$$
Hence the flag is non-splitting.
\end{proof}

The hypothesis $s\leq n$ excludes the case where $M(C)$ coincides with the matroid of the columns of $\overline{B}$, which we will consider in the next section. Indeed, $\spann_{\RR}\langle g_j:\:j\in[m]\rangle$ has dimension $m-n-1=\ell+(s-n-1)$, while a flag of length at most $\ell$ gives at most $\ell-1$ vectors $s_i$ in \eqref{eq: Gspanning}. However, in that case $\dim L_{E_i}^*=r_D(E_i)=i$, and the arguments in the proof above show that a complete flag is non-splitting if and only if $s_1,\dots,s_{\ell-1}$ are linearly independent. Using this, we recover the results in \cite{CURRAN2007115, cattani2022nonsplittingflagsiteratedcircuits} in \Cref{section:4}.

\begin{example}
\label[example]{ex: GaleDual}
The Gale dual of the vertical system in \Cref{example: pyramid} is $H = (D(y) \star \frac{1}{u})^G - \mathbf{1}$ with
	$$D^{\top} = \begin{pmatrix}
		-1 & -1 & 0 & 1 & 0 \\
		0 & -1 & -1 & 0 & 1 \\
	\end{pmatrix}, \quad G^{\top} = \begin{pmatrix}
        -1 & 0 & 1 & -1 & 1\\
	\end{pmatrix}$$
	which corresponds to the polynomial system:
    $$H = \frac{u_1u_4}{u_3u_5}\,y^2 - 1, \qquad y = y_2/y_1.$$
    In particular, the chain of flats $0 \subset E_1 \subset E_2$ with $E_1 = \{3,5\}$ satisfies that
    $$\sum_{j \in E_1}g_j = 1 + 1 \neq 0,$$
    implying that the flag is $G$-spanning and also non-splitting.
\end{example}

\subsection{Root counts and generic consistency}

Finally, we derive a consequence of the previous results for square systems. Namely, if a square system is not generically consistent, then the generic number of zeros in $\Tn$ must be zero. As these root counts have been studied in \cite{checa2026toricvectorbundlesnumber, kaveh2025vectorvaluedlaurentpolynomialequations,genericFlatness}, we may use \Cref{thm:mainTheorem} to state equivalent conditions for consistency of square systems.

 Let $F=C(u\star x^B)$ be a vertical system such that $\rk C=n$. Consider the following set of polytopes:
\begin{equation}\label{eq:Np}\Delta^{(i)} \coloneqq \conv \Big\{\sum_{k \in \mathcal{I}}b_{k} \in \mathbb{Z}^n : \:  \mathcal{I} \subset [m], \: |\mathcal{I}| = i, \: (c_{k})_{k \in \mathcal{I}} \text{ are linearly independent}\Big\} \subset \R^n
	\end{equation}
where $c_j$ is the $j$-th column of $C \in \CC^{n \times m}$. In particular, $\Delta^{(1)} = \Delta$.
In this case of square systems (when $s = n$), recent work of Kaveh, Khovanskii and Spink \cite{kaveh2025vectorvaluedlaurentpolynomialequations} has studied the generic number of roots over $\Tn$ by considering the mixed volume function of the virtual polytopes $\Delta^{(i + 1)} - \Delta^{(i)}$, where the virtual mixed volume is understood by linear extension of the mixed volume function. On the other hand, we let $X_B \subseteq \mathbb{P}^{m-1}$ be the toric variety defined by the toric ideal
\[
I_{\overline{B}} \coloneqq \big\langle\, y^a - y^c \: : \: a,c \in \ZZ_{\geq 0}^{m}, \: \overline{B}a = \overline{B}c \,\big\rangle \subset \mathbb{C}[y_1,\dots,y_m],
\]
which equals the closure of the image of
\begin{equation}
\label{eq: toricVariety}
    (\mathbb{C}^*)^n \xrightarrow[]{} \mathbb{P}^{m-1}, \quad t \xrightarrow[]{} [t^{b_1}:\dots:t^{b_m}].
\end{equation}
Let $Y_{C}$ be the variety defined by $\langle Cy  \rangle \subset \mathbb{C}[y_1,\dots,y_m]$.

\begin{corollary}
	\label[corollary]{thm: equivalentConditions}
	Let $F = C(u \star x^B)$ be a vertical system with $\rk(C) = n$. The following are equivalent:
	\begin{enumerate}[label=(\roman*)]
		\item The tropicalizations of $X_B$ and $Y_{C}$ have a nonempty (stable) intersection.
		\item There exists an iterated circuit $(\mathcal{C}_1,\dots,\mathcal{C}_{k})$ of $M(C)$ such that $$\dim \big( \sum_{i = 1}^{k} \Delta_{\mathcal{C}_i}\big)=n.$$
		\item The mixed volume $\MV(\Delta^{(1)},\Delta^{(2)} - \Delta^{(1)},\dots,\Delta^{(n)} - \Delta^{ (n-1)})$ is positive, where each entry is considered as a multilinear function.
	\end{enumerate}
\end{corollary}

\begin{proof}
	By \cite[Theorem 5.1]{kaveh2025vectorvaluedlaurentpolynomialequations}, \cite[Theorem 3.21]{feliu2026rootboundsverticalsystems} and \Cref{thm:mainTheorem}, the three conditions are equivalent to the system being generically consistent over $\Tn$.
\end{proof}

\begin{example}
	Consider the system $F = C(u \star x^B)$ with matrices
	$$C = \begin{pmatrix}
		1 & 1 & 1 \\
		0 & 1 & 1
	\end{pmatrix}, \quad B = \begin{pmatrix}
		0 & 1 & 0 \\
		0 & 0 & 1
	\end{pmatrix}.$$
	Note that the only circuit in $M(C)$ consists of the columns $\{2,3\}$. The Newton polytope $$\Delta = \Delta^{(1)} = \conv \Big\{(0,0),(1,0),(0,1)\Big\} \subset \mathbb{R}^2$$ equals the unit simplex. Moreover, as the second and third columns of $C$ are linearly dependent, $\Delta^{(2)} = \conv\Big\{(1,0),(0,1)\Big\}$. Therefore,
	$$\MV(\Delta^{(1)},\Delta^{(2)} - \Delta^{(1)}) = \MV(\Delta^{(1)},\Delta^{(2)}) - \MV(\Delta^{(1)},\Delta^{(1)}) = 1 - 1 = 0.$$
	Equivalently, it can be checked that the system has no iterated circuits $(\mathcal{C}_1, \mathcal{C}_2)$ such that $\dim  \big(\Delta_{\mC_1} +  \Delta_{\mC_2}\big)=2.$ Note that as the first column of $C$ is not contained in any circuit, we have $\Trop(Y_C) = \varnothing$.
\end{example}

\section{The resultant of a vertical system}
\label{section:4}

In the previous section, we saw that the iterated circuits of $M(C)$ can be used to study the properties of the vertically parametrized polynomial system $F = C(u \star x^B)$. If the system is generically consistent, then the projection of $\mathcal{W}_F$ to the parameter space is dense. In this section we provide a more general formula for the codimension of  this projection which, in particular, applies to the overdetermined systems. When this projection is defined by a single polynomial, we define the resultant of the vertical system and study some of its properties, in analogy with sparse resultants \cite{cannyemiris93, dandrea2020cannyemiris, sturmfels94}.

\subsection{The codimension of the projection to the parameter space}
Let $\mathcal{B}_1,\dots,\mathcal{B}_c\subseteq[m]$ be the connected
components of $M(C)$. By \Cref{thm: connectedComponentsInvariant}, we
consider the parameter space $\mathbb{P}^{\mathcal{B}}
\coloneqq
\mathbb{P}^{\mathcal{B}_1}\times\cdots\times
\mathbb{P}^{\mathcal{B}_c}$  and the incidence variety
\[
\mathcal{W}_F
\coloneqq
\{(x,u)\in\mathbb{T}^n\times\mathbb{P}^{\mathcal{B}}:
C(u\star x^B)=0\}.
\]
We denote by $\pi\colon\mathcal{W}_F\longrightarrow\mathbb{P}^{\mathcal{B}}$ the projection to the second factor. Note that $\mathcal{W}_F$ is the total space of the projectivization in $\PP^{\mathcal{B}}$ of the vector bundle given by the right  kernels of $C \diag(x^B)$ for $x \in \Tn$, as
$$C(u \star x^B) = C\diag(u)x^B = C\diag(x^B)u.$$
In particular, $\mathcal W_F$ is irreducible and
\[
\dim\mathcal W_F=n+m-s-c.
\]
The projection $\overline{\pi(\mathcal{W}_F)}$ is also irreducible. The proof of \Cref{thm:mainCodimension} will make use of the following two lemmas.

\begin{lemma}
\label{lemma: mC}
Let $\mathcal C=(\mathcal C_1,\dots,\mathcal C_k)$ be an iterated
circuit of $M(C)$. There exists $I\subseteq \bigcup_{i=1}^k\mathcal C_i$ such that, after omitting the empty
sets, the sequence
$(\mathcal C_1\setminus I,\dots,\mathcal C_k\setminus I)$
is an iterated circuit of $M(I)$ satisfying
\begin{equation}
\label{eq: rIdimDelta}
r_{I}\Big(\bigcup_{i=1}^k\mathcal C_i\setminus I\Big)
=
\dim\left(\sum_{\mathcal C_i\setminus I\neq\varnothing}
\Delta_{\mathcal C_i\setminus I}\right).
\end{equation}
\end{lemma}
\begin{proof}
We choose $j_i\in\mathcal C_i$ and subsets $K_i\subseteq\mathcal C_i\setminus\{j_i\}$ for each $i = 1,\dots,k$ such that the vectors in $\bigcup_{i = 1}^k\bigcup_{a \in K_i}\{b_a - b_{j_i}\}$ form a basis of the direction space of $\sum_i\Delta_{\mathcal C_i}$, and set $I=\bigcup_i(\mathcal C_i\setminus(K_i\cup\{j_i\}))$.
Since $\bigcup_i(\mathcal C_i\setminus\{j_i\})$ is independent in $M(C)$, the sets $K_i\cup\{j_i\}$ with $K_i\neq\varnothing$ (in their original order) form an iterated circuit of $M(I)$ satisfying \eqref{eq: rIdimDelta}.
\end{proof}

\begin{lemma}
\label{lemm: sameCircuits}
Let $C \in \mathbb{C}^{s \times m}$ be a matrix and let $M(C)$ be the corresponding matroid. Assume M(C) has no coloops. Given two elements $i,j \in [m]$, the following are equivalent:
    \begin{enumerate}[label=(\roman*)]
        \item For any circuit $\mathcal{C}$ of $M(C)$, $i \in \mathcal{C}$ if and only if $j \in \mathcal{C}$.
        \item There exists $\lambda \in \mathbb{C}^*$ such that $w_j = \lambda w_i$ for all $w \in \ker(C)$.
    \end{enumerate}
\end{lemma}

\begin{proof}
Assume (i) holds and consider the linear forms $L_i: \ker(C) \mapsto \CC$ to be the projection to the $i$-th coordinate which are nonzero since there are no coloops. If the operators $L_i$ and $L_j$ are not proportional, their kernels would be two different hyperplanes of $\ker(C)$, so some $w \in \ker(C)$ would have $w_j = 0 \neq w_i$. Choose such a $w$ with minimal support. Then, $w$ is a circuit vector: if $w'\in\ker(C)\setminus\{0\}$ has $\supp(w')\subsetneq\supp(w)$, then either $w'_i\neq0$ and $w'$ contradicts the minimality, or $w'_i=0$ and $w-\mu w'$ does, for $\mu$ chosen to cancel one coordinate of $\supp(w')$. Therefore, $\supp(w)$ would be a circuit containing $i$ but not $j$, getting a contradiction.

Conversely, if there exists $\lambda \in \mathbb{C}^*$ such that $w_j = \lambda w_i$ then every vector of $\ker(C)$ has either both $i,j$ or neither of them in its support.
\end{proof}

To a vertical system $F = C(u \star x^B)$, we attach the integer
\begin{equation}
\label{eq: iota}
    \iota(F)\coloneqq
    \max_{(\mathcal C_1,\dots,\mathcal C_k)}
    \dim\left(\sum_{j=1}^k\Delta_{\mathcal C_j}\right),
\end{equation}
where the maximum is taken over all iterated circuits of $M(C)$.

 \begin{theorem}
	\label[theorem]{thm:mainCodimension}
	Let $F = C(u \star x^B)$ be a vertically parametrized polynomial system with $\rk(C) = s$. If $\mathcal{W}_F \neq \varnothing$,  then
	\begin{equation}
		\label{codimension maximum}
		\codim(\overline{\pi(\mathcal{W}_F)}) = s - \iota(F).
	\end{equation}

\end{theorem}\begin{proof}
Let $\xi=(x,u)\in\mathcal W_F$ be a generic point and by generic smoothness, the fibre through $\xi$ is smooth of dimension $n-\rank J_xF(\xi)$. Furthermore, $\dim\mathcal W_F=n+m-s-c$ (where $c$ is the number of connected components of $M(C)$) and hence, by the fiber dimension theorem it suffices to show that $\rank J_xF(\xi)=\iota(F)$. Here, $J_xF(\xi)=C\diag(w)B^{\top}\diag(x)^{-1}$ with $w=u\star x^B\in\ker(C)$.

As $\mathcal W_F\neq\varnothing$, $M(C)$ has no coloops: if $i$ were a coloop, then $w_i=0$ for every $w\in\ker(C)$, so any $(x,u)\in\mathcal W_F$ would have $u_ix^{b_i}=0$ and hence $u_i=0$, while $\{i\}$ is a connected component and $\PP^{\mathcal B_i}=\PP^0$ forces $u_i\neq0$. Taking $x=\mathbf 1$, $w$ then ranges over a dense subset of $\ker(C)$. So we must show that the generic rank of $C\diag(w)B^{\top}$ on $\ker(C)$ is $\iota(F)$.

Take an iterated circuit attaining $\iota(F)$, and $I$, $Y=\bigcup_i\mC_i\setminus I$ as in \Cref{lemma: mC}. After removing linearly dependent rows, the system given by the columns $Y$ of $C$ is generically consistent by \Cref{thm:mainTheorem}. By \Cref{thm:genericConsistencyOSkarElisendaBeatriz}, its matrix $C_I\diag(w')B^{\top}$ has rank $\iota(F)$ for some $w'$ in its kernel. Extending $w'$ to $\mathbb{C}^m$ by adding zeros outside the columns of $C_I$ lift it to some $w\in\ker(C)$. Writing $C$ as in \eqref{equation:contraction}, these rows appear among the rows of $C\diag(w)B^{\top}$. Then $\rank C\diag(w)B^{\top} \geq \iota(F)$.

We show that the rank of $ C\diag(w)B^{\top}$ is less or equal than $\iota(F)$ for all $w\in\ker(C)$ and any $C$, by induction on $m$. The smallest $m$ for which we can have $\mathcal{W}_F \neq \varnothing$ is $m = s + 1$ and we only have $\mathcal{W}_F \neq \varnothing$ when $[m]$ is the unique circuit of $M(C)$. In that case, the rank equals the dimension of the affine span of $\Delta_{[m]}$ i.e. $\iota(F)$. Moreover, we may delete zero columns and coordinates vanishing on $\ker(C)$, and assume $\iota(F)<s$.

\begin{enumerate}[label=(\roman*)]
\item Suppose two elements of $M(C)$ lie in exactly the same circuits, let $\mathcal A \subset [m]$ be the set of all elements lying in the same circuits as them, and fix $i\in\mathcal A$. By \Cref{lemm: sameCircuits}, there are $v_j\neq0$ with $w_j=v_jw_i$ for all $w\in\ker(C)$ and $j\in\mathcal A$. Let $\widehat F=\widehat C(u\star x^{\widehat B})$, where $\widehat C$ is obtained from $C$ by replacing the columns in $\mathcal A$ by the single column $\sum_{j\in\mathcal A}v_jc_j$, and $\widehat B$ is $B_{\mathcal A}$ with the extra column $\tau_{\mathcal A}(b_i)$. If $\sum_{j\in\mathcal A}v_jc_j = 0$ the new column is a loop of $M(\widehat C)$; deleting this column does not change $\ker(\widehat C)$ or $\iota(\widehat F)$, so we may assume that $\widehat C$ has no zero column. Since $\tau_{\mathcal A}(b_j)=\tau_{\mathcal A}(b_i)$ for $j\in\mathcal A$,
\[
  \sum_{k=1}^m w_k\,c_k\,\tau_{\mathcal A}(b_k)^{\top}=\widehat C\diag(\widehat w)\widehat B^{\top},\qquad \widehat w=\big((w_k)_{k\notin\mathcal A},w_i\big)\in\ker(\widehat C).
\]
The left-hand side is $C\diag(w)B^{\top}$ composed with $\tau_{\mathcal A}^{\top}$. As $\ker\tau_{\mathcal A}=L_{\mathcal A}$, the rank of $C\diag(w)B^{\top}$ is at most $\dim\Delta_{\mathcal A}$ plus the rank of this matrix, hence at most $\dim\Delta_{\mathcal A}+\iota(\widehat F)$ by induction. Note that $w\mapsto\widehat w$ is an isomorphism between $\ker(C)$ and $\ker(\widehat C)$, and the relations $w_j=v_jw_i$ for $j\in\mathcal A$ are preserved when we contract by a set not intersecting $\mathcal A$. Since the circuits of a represented matroid are the minimal supports of nonzero kernel vectors and, by \Cref{lemm: sameCircuits}, the support of every $w\in\ker(C)$ contains either all of $\mathcal A$ or none of it, the isomorphism $w\mapsto\widehat w$ carries $\supp(w)$ to the set obtained from it by replacing $\mathcal A$ by the new column, and hence induces a bijection between the circuits of $M(C)$ and those of $M(\widehat C)$. As the relations $w_j=v_jw_i$ persist after contracting by a set disjoint from $\mathcal A$, the same bijection holds for all the successive contractions. Therefore, the iterated circuits of $M(\widehat C)$ are the iterated circuits $\mC$ of $M(C)$ where we replace $\mathcal A$ by the new column. As $\tau_{\mathcal A}(b_j)=\tau_{\mathcal A}(b_i)$ for all $j\in\mathcal A$, the corresponding direction spaces of $\widehat F$ are $\tau_{\mathcal A}(L_{\mC_t})$, implying that such an iterated circuit has value
\[
\dim\tau_{\mathcal A}(\Sigma)=\dim(\Sigma+L_{\mathcal A})-\dim\Delta_{\mathcal A}, \quad \text{where } \Sigma\coloneqq\sum_tL_{\mC_t}.
\]
If some $\mC_i$ in $\mC$ contains $\mathcal A$, this is at most $\iota(F)-\dim\Delta_{\mathcal A}$. Otherwise, $\mathcal A$ lies in a further circuit $\mathcal{C}_{k + 1}$ of the contraction with
$$\dim(\Sigma+L_{\mathcal C_{k+1}})\ge\dim(\Sigma+L_{\mathcal A}).$$
Hence, $\iota(\widehat F)\le\iota(F)-\dim\Delta_{\mathcal A}$.

\item Suppose finally that no two elements lie in exactly the same circuits. By \Cref{lemm: sameCircuits}, the $m$ linear forms $L_j$ on $\ker(C)$ are nonzero and pairwise non-proportional for $j \in [m]$. For each $j \in [m]$, the hyperplane $T_j=\{w\in\ker(C):w_j=0\}$ is the kernel of the matrix obtained from $C$ by deleting the column $j \in [m]$, and for $w\in T_j$ the matrix $C\diag(w)B^{\top}$ is the one of the vertical system $F_j$ obtained from $F$ by deleting the column $j$. Every iterated circuit of the matroid of $F_j$ is an iterated circuit of $M(C)$, so $\iota(F_j)\le\iota(F)$. By induction, $\rank C\diag(w)B^{\top}\le\iota(F)$ for all $w\in T_j$. Now let $P$ be an $(\iota(F)+1)$-minor of $C\diag(w)B^{\top}$. As in the proof of \Cref{lemma: derivativeCircuit}, $P$ vanishes on every $T_j$. Therefore, the product of the $m$ pairwise non-proportional forms $L_j$ divides $P$. Since $\deg P=\iota(F)+1\le s<m$, we get $P=0$.
\end{enumerate}

\end{proof}

\begin{corollary}
\label[corollary]{prop: codimFlags}
Let $F=C(u\star x^B)$ be a vertical system with $ \rk(C)=s$ and $\mathcal W_F\neq\varnothing$. For a complete flag $\varnothing=E_0\subset E_1\subset\dots\subset E_\ell=[m]$ of flats of $M(D)$ with $r_D(E_i)=i$, let $s_1,\dots,s_\ell$ be as in \eqref{eq: si}. Then
\[
\codim\big(\overline{\pi(\mathcal W_F)}\big)=(m-n-1)-\max\,\dim\spann_{\RR}\langle s_1,\dots,s_{\ell-1}\rangle,
\]
where the maximum is taken over all such complete flags.
\end{corollary}

\begin{proof}
By \Cref{thm:mainCodimension}, the codimension equals $s-\iota(F)$. Since $\mathcal W_F\neq\varnothing$, $M(C)$ has no coloops, so an iterated circuit whose union has rank smaller than $s$ can be extended by a circuit of the contraction with at least two elements. This does not decrease the dimension of the Minkowski sum, so $\iota(F)$ is attained by an iterated circuit whose union has rank $s$. For such an iterated circuit and the complete flag \eqref{eq: completeFlag}, the proof of \Cref{thm: GspanningFlag} shows that $\dim\big(\sum_i\Delta_{\mathcal C_i}\big)=s-\dim(U\cap\ker\overline B)$ and that $U\cap\ker\overline B$ is isomorphic to $\{y\in\RR^{m-n-1}:\langle s_i,y\rangle=0\text{ for all }i\}$. Hence
\[
\dim\Big(\sum_i\Delta_{\mathcal C_i}\Big)=s-(m-n-1)+\dim\spann_{\RR}\langle s_1,\dots,s_{\ell}\rangle.
\]
Conversely, as in the proof of \Cref{thm: GspanningFlag}, every complete flag arises in this way from an iterated circuit whose union has rank $s$, after omitting the sets $E_i\setminus E_{i-1}$ with one element. Finally, $\sum_is_i=\sum_jg_j=0$, so $s_\ell$ lies in the span of $s_1,\dots,s_{\ell-1}$. Taking the maximum over all complete flags proves the claim.
\end{proof}

\subsection{Defective toric varieties}

 \Cref{thm:mainCodimension} generalizes known results on defectivity of toric varieties. In particular, if $B \in \mathbb{Z}^{n \times m}$ is a matrix, consider the system
    \begin{equation}
    \label{eq: FB}
    F = (F_B,x_1\frac{\partial F_B}{\partial x_1},\dots,x_n\frac{\partial F_B}{\partial  x_n}), \quad F_B = \sum_{i = 1}^m u_ix^{b_i}.
    \end{equation}
    We note that the system $F = \overline{B}(u \star x^B)$ is vertical where $\overline{B} \in \ZZ^{(n + 1) \times m}$ and let $G \in \mathbb{Z}^{m \times \ell}$ be the Gale dual of $\overline{B}$. We say that the toric variety $X_B$ (as in \eqref{eq: toricVariety}) is \textit{defective} if $\codim (\overline{\pi(\mathcal{W}_F)}) > 1$. The following theorem recovers \cite[Theorem 18]{CURRAN2007115} and \cite[Theorem 3.8]{cattani2022nonsplittingflagsiteratedcircuits}, using our previous results.

    \begin{theorem}
\label{thm: dualDefects}
        Let $B \in \mathbb{Z}^{n \times m}$ be a matrix and $F$ be the vertical system in \eqref{eq: FB}. The following are equivalent:
        \begin{enumerate}[label=(\roman*)]
            \item $X_B$ is defective.
            \item For all iterated circuits  $\mathcal{C} = (\mathcal{C}_1,\dots,\mathcal{C}_{k})$ in $M(\overline{B})$, we have
            $$\dim\big(\sum_{i = 1}^{k} \Delta_{\mathcal{C}_i}\big) < n.$$
            \item All non-splitting flags in $M(G)$ have length $< \ell$.
            \item For every complete flag of flats in $M(G)$, the vectors $s_1,\dots,s_{\ell-1}$ in \eqref{eq: si} are linearly dependent.
        \end{enumerate}
    \end{theorem}

\begin{proof}
    The equivalence between (i) and (ii) follows from \Cref{thm:mainCodimension}. Here, $C=\overline{B}$, so $s=n+1$ and $\ell=m-n-1$, and we may take $D=G$. In particular, $M(D)=M(G)$ and $\dim L_E^*=r_G(E)$ for every $E\subseteq[m]$. If $\mathcal W_F=\varnothing$, then $M(G)$ has loops, so there are no complete flags with $E_0=\varnothing$, and (i), (iii) and (iv) hold. Otherwise, by \Cref{prop: codimFlags}, the codimension of $\overline{\pi(\mathcal W_F)}$ equals $\ell-\max\dim\spann_{\RR}\langle s_1,\dots,s_{\ell-1}\rangle$. So it is larger than one if and only if $s_1,\dots,s_{\ell-1}$ are linearly dependent for every complete flag, proving that (i) and (iv) are equivalent. Finally, a flag of length $\ell$ is complete, and since $\dim L_{E_i}^*=i$, the arguments in the proof of \Cref{cor: nonsplitGspan} show that a complete flag is non-splitting if and only if $s_1,\dots,s_{\ell-1}$ are linearly independent. This proves the equivalence between (iii) and (iv).
\end{proof}

\subsection{The resultant of a vertical system}

As a consequence of \Cref{thm:mainCodimension}, we can combinatorially describe the codimension of $\overline{\pi(\mathcal{W}_F)}$. This applies also to the case where this codimension equals $1$, which we can further analyze in the following theorem.

\begin{theorem}
	\label[theorem]{thm: codimensionOne}
	Let $F = C(u \star x^B)$ be a vertical system and assume that $F$ is not generically consistent and $\mathcal{W}_F \neq \varnothing$. The projection $\overline{\pi(\mathcal{W}_{F})}$ has codimension $1$ if and only if there exists $i \in [m]$ such that the system $F_{(i)} = C_{(i)}(u \star x^B)$ is generically consistent.
\end{theorem}

\begin{proof}
Suppose first that $\overline{\pi(\mathcal W_F)}$ has codimension $1$. By \Cref{thm:mainCodimension}, we have
$\iota(F)=s-1$. Choose an iterated circuit
$\mathcal C=(\mathcal C_1,\ldots,\mathcal C_k)$
such that
$\dim\left(\sum_{j=1}^k\Delta_{\mathcal C_j}\right)=s-1.$ Since $\mathcal W_F\neq\varnothing$, $M(C)$ has no coloops, so we may extend $\mathcal C$ until $r_C(\bigcup_j\mathcal C_j)=s$, and the dimension of the Minkowski sum remains $s-1$.

Choose $j_t\in\mathcal C_t$ and $K_t\subseteq\mathcal C_t\setminus\{j_t\}$ as in the proof of \Cref{lemma: mC}, so that $\sum_t|K_t|=s-1$. Then $\mathcal D=\bigcup_t(\mathcal C_t\setminus\{j_t\})$ is a basis of $M(C)$ with exactly one element $i\notin\bigcup_tK_t$, say $i\in\mathcal C_{t_0}$. As $\mathcal D$ is independent, contracting $i$ does not change $\mathcal C_1,\dots,\mathcal C_{t_0-1}$, and $\mathcal C_{t_0}\setminus\{i\}$ is a circuit of $M(\mathcal C_1\cup\dots\cup\mathcal C_{t_0-1}\cup\{i\})$. Omitting it if it has one element, we obtain an iterated circuit of $M(\{i\})$ whose Minkowski sum contains the vectors $b_a-b_{j_t}$, $a\in K_t$, and hence has dimension $s-1$. By \Cref{thm:mainTheorem}, $F_{(i)}$ is generically consistent.

Conversely, suppose that $F_{(i)}$ is generically consistent for some
$i\in[m]$. The system $F$ is obtained from $F_{(i)}$ by imposing one
additional equation, which is the only one containing $u_i$ when $C$ is written in row-reduced echelon form with respect to a basis containing $i$. So, for generic $(u_j)_{j\neq i}$ and a zero $x$ of $F_{(i)}$, this equation can be solved for $u_i$. Hence $\overline{\pi(\mathcal W_F)}$ has codimension at most 1. Since $F$ is not generically consistent,
$\overline{\pi(\mathcal W_F)}$ is a proper subvariety of the parameter
space, and therefore has positive codimension. It follows that
$\overline{\pi(\mathcal W_F)}$ has codimension $1$.
\end{proof}

As we know that $\overline{\pi(\mathcal{W}_F)}$ is irreducible, we may proceed with the following definition.
\begin{definition}
	\label[definition]{def: resultant}
	Let $F = C(u \star x^B)$ be a vertically parametrized system such that $s - \iota(F) = 1$. We call a \emph{resultant} $\Res_F$ of the vertical system $F$ to be any irreducible polynomial in $\mathbb{C}[u]$ that vanishes over $\overline{\pi(\mathcal{W}_F)}$ in $\mathbb{P}^{\mathcal{B}}$.
\end{definition}

We are aware that, in the terminology of \cite{gkz1994}, the sparse resultant is defined as a unique (up to sign) irreducible polynomial with coefficients in $\mathbb{Z}$ with the gcd of the coefficients being $1$. In our case, $\Res_F$ is defined up to scalar and the rest of statements in this section do not depend on this scalar. The following result describes the parameters on which $F$ depends, generalizing \cite[Corollary 1.1]{sturmfels94}.

\begin{corollary}
	\label[corollary]{corollary: parametersThatAppear}
	Let $F = C(u \star x^B)$ be a vertical system and assume that $\codim (\overline{\pi(\mathcal{W}_F)}) = 1$. For each $i \in \{1,\dots,m\}$, the following are equivalent:
	\begin{enumerate}[label=(\roman*)]
		\item $\Res_F \in \mathbb{C}[u_j : j \neq i]$;
		\item The system $F_{(i)}$ is generically inconsistent.
	\end{enumerate}
\end{corollary}

\begin{proof}
We denote by $\mathcal{W}_i$ the incidence variety of $F_{(i)}$. If the system $F_{(i)}$ is generically inconsistent, then $\overline{\pi(\mathcal{W}_i)}$is not the whole parameter space. As a consequence,
	\begin{enumerate}[label=(\roman*)]
		\item There exists an irreducible polynomial $p \in \mathbb{C}[u_j : \: j \neq i]$ that vanishes in $\overline{\pi(\mathcal{W}_{i})}$, and hence on $\overline{\pi(\mathcal{W}_F)}$, since the equations of $F_{(i)}$ are among those of $F$ after row reduction with respect to the column $i$.
		\item As $\Res_F$ is irreducible, we have that the variable $u_i$ does not divide $\Res_F$.
	\end{enumerate}
	In particular, if $\Res_F \notin \mathbb{C}[u_j : \: j \neq i]$, then $\Res_F$ is a nonzero divisor in $\mathbb{C}[u]/\langle p\rangle$. As both $\Res_F$ and $p$ vanish on $\overline{\pi(\mathcal{W}_{F})}$, $\codim (\overline{\pi(\mathcal{W}_{F})}) \geq 2$, leading to a contradiction. Conversely, if $\Res_F \in \mathbb{C}[u_j : \: j \neq i]$, then $\overline{\pi(\mathcal{W}_{i})}$ is not the whole parameter space: otherwise, as in the proof of \Cref{thm: codimensionOne}, for generic $(u_j)_{j\neq i}$ there is $u_i$ leading to $u\in\pi(\mathcal W_F)$, and $\Res_F$ would vanish identically. This implies that the system $F_{(i)}$ is not generically consistent.
\end{proof}

\begin{example}
	Consider the following system:
	\[F=\{u_1 x y + u_2 x,u_3 z w + u_4 z,u_5 y + u_6 w\}.\]
	The corresponding $C$ and $B$ matrices are given by:
	\begin{equation}
		\begin{split}
			C = \begin{pmatrix}
				1 & 1 & 0 & 0 & 0 & 0 \\
				0 & 0 & 1 & 1 & 0 & 0\\
				0 & 0 & 0 & 0 & 1 & 1
			\end{pmatrix}, \quad B = \begin{pmatrix}
				1 & 1 & 0 & 0 & 0 & 0\\
				1 & 0 & 0 & 0 & 1 & 0 \\
				0 & 0 & 1 & 1 & 0 & 0 \\
				0 & 0 & 1 & 0 & 0 & 1
			\end{pmatrix}
		\end{split}
	\end{equation}
	The circuits of $M(C)$ are given by $\mC_1:=\{1,2\},\mC_2:=\{3,4\},$ and $\mC_3:=\{5,6\}$ and any iterated circuit of $M(C)$ is given by a subset of $\{\mC_1,\mC_2,\mC_3\}.$ From this it is easy to check
	$\iota(F)=2$ and $\codim(\overline{\pi(\mathcal{W}_{F})})=3-2=1.$ In fact, $\overline{\pi(\mathcal{W}_{F})}$ is generated by $u_2u_3u_5+u_1u_4u_6=0.$
\end{example}

To conclude this section, we record two situations in which the vertical
resultant can be obtained from familiar constructions. Determining the degree of $\Res_F$ requires a finer comparison between resultants and
eliminants (for instance, \cite{dandrea2020cannyemiris}). We leave
this question for future work.

First, suppose that $s=n+1$. Let $\Res_B$ denote the unmixed sparse
resultant associated with the exponent configuration given by the
columns of $B$. Thus, $\Res_B$ is the resultant of the system
\[
F_0=\cdots=F_n=0,
\qquad
F_i=\sum_{j=1}^m\bar u_{i,j}x^{b_j},
\qquad
i=0,\ldots,n.
\]
By \cite[Theorem~1.1]{dalbec1995introduction}, $\Res_B$ is the Chow
form of $X_B\subseteq\mathbb P^{m-1}$. In particular, there exists a
polynomial
\[
P_B\in\mathbb Z[t_I:
I\subseteq[m],\ |I|=n+1]
\]
such that
\begin{equation}
\label{eq: chowForm}
\Res_B=P_B\bigl([j_0\cdots j_n]\bigr), \quad [j_0\cdots j_n]
\coloneqq
\det
\begin{pmatrix}
\bar u_{0,j_0} & \cdots & \bar u_{0,j_n}\\
\vdots & & \vdots\\
\bar u_{n,j_0} & \cdots & \bar u_{n,j_n}
\end{pmatrix}.
\end{equation}

For $I\subseteq[m]$, we write $u_I\coloneqq\prod_{j\in I}u_j$.

\begin{proposition}
\label[proposition]{prop: sparseResultant}
Let $F=C(u\star x^B)$ be a vertical system with
$C\in\mathbb C^{(n+1)\times m}$ of rank $n+1$, and assume that
$\overline{\pi(\mathcal W_F)}$ has codimension one. If $P_B\bigl(\det(C_I)u_I\bigr)\neq0$ and $\deg(\Res_B)=\deg(\Res_F)$, then
\[
\Res_F=P_B\bigl(\det(C_I)u_I\bigr),
\]
up to a nonzero scalar, where $C_I$ denotes the submatrix of $C$ with columns indexed by $I \subset [m]$.
\end{proposition}

\begin{proof}
Let $u\in\mathbb C^m$ satisfy $\Res_F(u)=0$. Then $F_u$ has a solution
in $\mathbb T^n$. Consider the specialization
\[
\bar u_{i,j}\coloneqq c_{i,j}u_j.
\]
The resulting sparse system is precisely $F_u$, and hence
\[
\Res_B\bigl((c_{i,j}u_j)_{i,j}\bigr)=0.
\]
For $I=\{j_0,\ldots,j_n\}$, we have
\[
[j_0\cdots j_n]
=
\det
\begin{pmatrix}
c_{0,j_0}u_{j_0} & \cdots & c_{0,j_n}u_{j_n}\\
\vdots & & \vdots\\
c_{n,j_0}u_{j_0} & \cdots & c_{n,j_n}u_{j_n}
\end{pmatrix}
=
\det(C_I)u_I.
\]
Thus $\Res_F$ divides $P_B(\det(C_I)u_I)$. Since these two polynomials
have the same degree, they are equal up to a nonzero scalar.
\end{proof}

Secondly, we consider the case in which the corank
$\ell=m-s$ of $C$ is small. The following formulas extend the familiar
low-codimension formulas for $\mathcal A$-discriminants in
\cite[Chapter~9, Proposition~1.8]{gkz1994} and
\cite[Theorem~4.2]{dickenstein2002elimination}.

\begin{proposition}
\label[proposition]{prop: lowcodimensionComputation}
Let $F=C(u\star x^B)$ be a vertical system such that
$\overline{\pi(\mathcal W_F)}$ has codimension one.

\begin{enumerate}[label=(\roman*)]
\item Suppose that $\ell=1$. Let
$D \in \mathbb{C}^{m \times 1}$ be a Gale dual of $C$, and let $G \in \mathbb{Z}^{m \times 1}$
be a Gale dual of $\overline B$. Then,
\begin{equation}
\label{eq: ResFl1}
\Res_F=
\prod_{g_j>0}d_j^{g_j}
\prod_{g_j<0}u_j^{-g_j}
-
\prod_{g_j<0}d_j^{-g_j}
\prod_{g_j>0}u_j^{g_j}.
\end{equation}

\item Suppose that $\ell=2$, that $G=(g_{ji})\in\mathbb Z^{m\times2}$,
and that no two columns of
\[
D=
\begin{pmatrix}
d_{11} & \cdots & d_{1m}\\
d_{21} & \cdots & d_{2m}
\end{pmatrix}
\]
are proportional. For $i=1,2$, define
\begin{equation}
\label{eq: ResFl2}
h_i(y)
\coloneqq
\prod_{g_{ji}>0}(d_{1j}+d_{2j}y)^{g_{ji}}
\prod_{g_{ji}<0}u_j^{-g_{ji}}
-
\prod_{g_{ji}<0}(d_{1j}+d_{2j}y)^{-g_{ji}}
\prod_{g_{ji}>0}u_j^{g_{ji}}.
\end{equation}
Then there exists $e\in\mathbb Z_{>0}$ such that
\[
\Res_y(h_1,h_2)=\Res_F^e,
\]
up to multiplication by a nonzero scalar and a monomial in $u$.
\end{enumerate}
\end{proposition}

\begin{proof}
We treat the two cases separately.

\begin{enumerate}[label=(\roman*)]
\item Since $\ell=1$, the Gale-dual system consists of the binomial
equation
\[
\prod_{g_j>0}
\left(\frac{d_j}{u_j}\right)^{g_j}
-
\prod_{g_j<0}
\left(\frac{d_j}{u_j}\right)^{-g_j}
=0.
\]
After clearing denominators, this is the polynomial in
\eqref{eq: ResFl1}. The entries of $G$ have greatest common divisor one,
so this binomial is irreducible. Its zero set is the resultant
hypersurface, proving the claim.
\item Without loss of generality, we assume that $d_{2j}$ $\neq 0$ for all $j\in[m]$ and let $
y_j=-\frac{d_{1j}}{d_{2j}}$. The common zero set of $h_1$ and $h_2$ in
\[
\mathbb C\setminus\{y_1,\ldots,y_m\}
\]
is the zero set of the Gale-dual system. Hence, by Gale duality, its
projection to the parameter space is the resultant hypersurface
$V(\Res_F)$.

We claim that no boundary point $y_i$ contributes a hypersurface factor
to $\Res_y(h_1,h_2)$. Since no two columns of $D$ are proportional,
\[
d_{1j}+d_{2j}y_i\neq0
\qquad\text{for every }j\neq i.
\]
If $g_{ia}\neq0$ for some $a\in\{1,2\}$, then exactly one of the two
monomials defining $h_a(y_i)$ vanishes, whereas the other is nonzero.
Thus, $h_a(y_i)\neq0$. If $g_i=0$, then $h_1(y_i)$ and $h_2(y_i)$ are Laurent binomials in $u$
whose exponent differences are the two independent columns of $G$.
Their common zero locus has codimension two. Therefore the boundary
points $y_1,\ldots,y_m$ do not contribute a codimension-one component
to the Sylvester resultant. The zero locus of $\Res_y(h_1,h_2)$ in $(\CC^*)^m$ consequently has the unique
codimension-one component $V(\Res_F)$. Since $\Res_F$ is irreducible,
there exists $e\in\mathbb Z_{>0}$ such that
\[
\Res_y(h_1,h_2)=\Res_F^e
\]
up to multiplication by a nonzero scalar and a monomial in $u$. \qedhere
\end{enumerate}
\end{proof}

The framework introduced in
\cite{esterov2026engineeredcompleteintersectionsalgorithmic} for
generalizing mixed subdivisions to vertical systems may provide, in
future work, subdivision-based formulas for resultants of vertical systems
analogous to the known formulas for sparse resultants
\cite{cannyemiris93,dandrea2020cannyemiris}.
\subsection{Iterated circuits and linear degeneracy}
\label{subsection: linearDegeneracy}

One of the key motivations for studying these systems arises from reaction network theory. Vertically parametrized systems play a central role in studying the properties of these networks. Given a reaction network, the associated ODEs that describe the evolution of species in the network over time are given by parametrized polynomials. These polynomials give a vertical system. For these systems, we are often interested in determining if the dynamics are robust at the steady state for certain species, i.e. the solution of the system is contained in the hyperplane parallel to one of the coordinate hyperplanes. This property in the literature is called absolute concentration robustness (ACR). We refer to \cite{localglobal,feliu2025genericgeometrysteadystate, Kaihnsa_Nguyen_Shiu_2024} for further study of this property in the biological context.

Although in biochemistry, this property is only relevant to studying the variety over $\mathbb{R}_{>0}^n$ and when the hyperplanes are parallel to the coordinate axes, in this section, we give definitions and results in further generality. In particular, for a given vertically parametrized polynomial system, we are interested in studying when the variety $V_{\Tn}(F_{u}) \subset \mathbb{C}^n$ is contained in a finite union of parallel hyperplanes.

\begin{definition}
	\label[definition]{def: linearlyDegenerate}
	A vertically parametrized polynomial system $F = C(u \star x^B)$ is \emph{linearly degenerate} with respect to $x_i$ if there exist $a_1,\dots,a_k \in \mathbb{C}^*$ such that $V_{\Tn}(F_{u}) \subset \bigcup_{j = 1}^k \{x_i = a_j\}$ for generic values of $u \in \mathbb{C}^m$.
\end{definition}

For $i \in \{1,\dots,n\}$, let $B_i \in \mathbb{Z}^{(n-1) \times m}$ be the matrix $B$ after removing the $i$-th row and let $\tau_i: \mathbb{R}^n \xrightarrow{} \mathbb{R}^{n-1}$ be the projection $\tau_i(x_1,\dots,x_n) = (x_1,\dots,x_{i-1},x_{i+1},\dots,x_n)$. We consider the vertical system given by $\widehat{F}_i = C(\overline{u} \star x^{B_i})$, where $\overline{u}=(\overline{u}_1,\ldots,\overline{u}_m)\in \CC^m$. Let $\mathcal{W}_i$ be the incidence variety of $\widehat{F}_i$. If $\codim ( \overline{\pi(\mathcal{W}_i)}) = 1$, we consider the polynomial $\Res_{\widehat{F}_i}(\overline{u})$ vanishing on $\overline{\pi(\mathcal{W}_i)}$. We denote by \( R_i(x_i, u) \) the polynomial in $ \mathbb{C}[x_i^{\pm}, u_1, \dots, u_m]$ obtained from \( \operatorname{Res}_{\widehat{F}_i}(\overline{u}) \) via the substitution
\[
\overline{u}_j = u_j x_i^{b_{ij}}, \qquad j = 1, \dots, m.
\]
That is, we set
\begin{equation}
\label{eq:Ri}
R_i(x_i, u) := \Res_{\widehat{F}_i}\big(u_1 x_i^{b_{i1}}, \dots, u_m x_i^{b_{im}}\big),
\end{equation}
for fixed exponents \( b_{ij} \in \mathbb{Z} \). If $\codim(\overline{\pi(\mathcal{W}_i)}) > 1$, then we set $R_i \coloneqq 0$.

With this definition in place (\eqref{eq:Ri}), we prove the following result, which gives a criterion for linear degeneracy in the variable \( x_i \).

\begin{theorem}
\label[theorem]{thm: linearDegeneracy}
Let \(F=C(u\star x^B)\) be a generically consistent vertical system with
\(\operatorname{rank}(C)=s\). The following are equivalent:
\begin{enumerate}[label=(\roman*)]
    \item The vertically parametrized family \(F\) is linearly
    degenerate with respect to \(x_i\).

    \item For every iterated circuit
    \((\mathcal{C}_1,\dots,\mathcal{C}_k)\) of \(M(C)\),
    \[
    \dim\left(
        \sum_{j=1}^k\tau_i(\Delta_{\mathcal{C}_j})
    \right)<s.
    \]

    \item The polynomial \(R_i\) is nonzero and $\langle F\rangle\cap\CC(u)[x_i^{\pm}]=\langle R_i\rangle$.
\end{enumerate}
\end{theorem}

\begin{proof}
By \cite[Theorem~4.4]{feliu2025genericgeometrysteadystate}, the family
\(F\) is linearly degenerate with respect to \(x_i\) if and only if
\(\widehat{F}_i\) is generically inconsistent. By
\Cref{thm:mainTheorem}, this is equivalent to
\[
\dim\left(
    \sum_{j=1}^k\tau_i(\Delta_{\mathcal{C}_j})
\right)<s
\]
for every iterated circuit
\((\mathcal{C}_1,\dots,\mathcal{C}_k)\) of \(M(C)\). This proves the
equivalence of (i) and (ii). Moreover, \cite[Theorem~4.4]{feliu2025genericgeometrysteadystate}
shows that these conditions are equivalent to the elimination ideal
\(\langle F\rangle\cap\CC(u)[x_i^{\pm}]\) being generated by one nonconstant polynomial $g$. If (iii) holds, then $R_i$ is such a generator, as $\langle F\rangle\neq(1)$ since $F$ is generically consistent. Conversely, assume (i). For $x\in\Tn$, let $x'$ be $x$ without its $i$-th coordinate and $\overline u=(u_jx_i^{b_{ij}})_j$, so that $C(u\star x^B)=C(\overline u\star x'^{B_i})$. The map $(u,x_i)\mapsto\overline u$ has one-dimensional fibres, $V_{\Tn}(F_u)$ is nonempty for generic $u$, and $\widehat F_i$ is not generically consistent. Hence $\overline{\pi(\mathcal W_i)}$ has codimension one, and $R_i$ is well-defined. For generic $u$, the zeros of $R_i(\cdot,u)$ are exactly the values of $x_i$ on $V_{\Tn}(F_u)$, since the points of $\overline{\pi(\mathcal W_i)}$ outside $\pi(\mathcal W_i)$ form a subset of codimension at least two. Moreover, $R_i$ is squarefree, being the pullback of the irreducible polynomial $\Res_{\widehat F_i}$ under the smooth map $(u,x_i)\mapsto\overline u$, and $g$ is squarefree, since for generic $u$ the zeros of $F_u$ are nondegenerate \cite[Theorem~3.7]{FELIU2025630}, so that $\langle F\rangle$ is radical. Hence $R_i$ and $g$ agree up to a unit. This proves the equivalence with (iii).
\end{proof}

We illustrate with the following examples the computations of $R_i$ and how to identify the planes that contain the variety $V_{\Tn}(F_u)$. The following example considers the system corresponding to the smallest network that exhibits both non-degenerate multistationarity and ACR \cite[Example 4.10]{Kaihnsa_Nguyen_Shiu_2024}.
\begin{example}
\label[example]{ex: reactionNetwork}
We consider the following polynomial system:
\begin{align*}
F_1:=	&\frac{dx_1}{dt} ~=~
	-u_1x_1x_2 + u_2 x_3\\
F_2:=	&\frac{dx_2}{dt} ~=~
	-u_1x_1x_2 + 2u_3 x_3^2\\
F_3:=	&\frac{dx_3}{dt} ~=~
	2u_1x_1x_2 - u_2x_3 - 2u_3x_3^2.
\end{align*}
The system above describes an ODE system that models evolution of concentrations of some biochemical species $X_1,X_2,$ and $X_3$. In particular, the above system arises from the following network:
\[
	G~=~
	\{X_1+X_2 \xrightarrow{u_1} 2X_3, \quad X_3\xrightarrow{u_2} X_1 , \quad 2X_3\xrightarrow{u_3} 2X_2\}~.
	\]
    For more details on the reaction networks aspects, the reader is referred to \cite{dickenstein2016biochemical}.
Since there is a linear relation between $F_1,F_2,$ and $F_3$, we only consider the system given by $F:=\{F_1,F_2\}$. This system is linearly degenerate in $x_3$. Here, the $C$ and $B$ matrices are:
\begin{equation}
	\begin{split}
		C = \begin{pmatrix}
			-1 & 1 & 0 \\
			-1 & 0 & 2\\
		\end{pmatrix}, \quad B = \begin{pmatrix}
			1 & 0 & 0\\
			1 & 0 & 0\\
			0 & 1 & 2
		\end{pmatrix}
	\end{split}
\end{equation}
So the system $\widehat{F}_3$ is generated by the matrices $C$ and $B_3$ obtained by removing the third row in the matrix $B$. In this case, the system $\widehat{F}_3$ is given by $\{-\overline{u}_1x_1x_2+\overline{u}_2,-\overline{u}_1x_1x_2+2 \overline{u}_3\}$. The resultant of these two equations is given by $\overline{u}_2-2\overline{u}_3$. Hence, the polynomial $R_3(x_3,u)$ is $u_2 x_3 - 2 u_3 x_3^2$. This shows that the solution of $V_{\Tn}(F_u)$ is contained in the plane $x_3=\frac{u_2}{2u_3}$ for generic values of $u$.
\end{example}

\printbibliography[]

@article{ArdilaPostnikov2010,
  author    = {Ardila, F. and Postnikov, A.},
  title     = {Combinatorics and geometry of power ideals},
  journal   = {Trans. Amer. Math. Soc.},
  volume    = {362},
  number    = {8},
  pages     = {4357--4384},
  year      = {2010}
}

@article{bihan2021optimaldescartes,
  author  = {Bihan, F. and Dickenstein, A. and Forsg{\aa}rd, J.},
  title   = {Optimal {D}escartes' rule of signs for systems supported on circuits},
  journal = {Math. Ann.},
  volume  = {381},
  number  = {3-4},
  pages   = {1283--1307},
  year    = {2021}
}

@article{bender2021koszultype,
  author  = {Bender, M. R. and Faug{\`e}re, J.-C. and Mantzaflaris, A. and Tsigaridas, E.},
  title   = {Koszul-type determinantal formulas for families of mixed multilinear systems},
  journal = {SIAM J. Appl. Algebra Geom.},
  volume  = {5},
  number  = {4},
  pages   = {589--619},
  year    = {2021}
}

@inproceedings{cannyemiris93,
  author    = {Canny, J. F. and Emiris, I. Z.},
  title     = {An efficient algorithm for the sparse mixed resultant},
  booktitle = {Appl. Algebra Algebraic Algorithms Error-Correcting Codes},
  volume    = {673},
  pages     = {89--104},
  year      = {1993}
}

@techreport{cannypedersen,
  author      = {Canny, J. F. and Pedersen, P.},
  title       = {An algorithm for the {N}ewton resultant},
  institution = {Cornell University},
  address     = {USA},
  year        = {1993}
}

@article{cattani2022nonsplittingflagsiteratedcircuits,
  author  = {Cattani, E. and Dickenstein, A.},
  title   = {Non-splitting flags, iterated circuits, $\sigma$-matrices and {C}ayley configurations},
  journal = {Vietnam J. Math.},
  volume  = {50},
  pages   = {1--28},
  year    = {2022}
}

@article{mixeddiscriminants,
  author  = {Cattani, E. and Cueto, M. and Dickenstein, A. and Di Rocco, S. and Sturmfels, B.},
  title   = {Mixed discriminants},
  journal = {Math. Z.},
  volume  = {274},
  year    = {2011}
}

@inproceedings{checaemiris,
  author    = {Checa, C. and Emiris, I. Z.},
  title     = {A greedy approach to the {C}anny-{E}miris formula},
  booktitle = {Proceedings of ISSAC '22},
  pages     = {283--291},
  publisher = {ACM},
  address   = {New York, NY, USA},
  year      = {2022}
}

@misc{checa2025effectivecriterionmultiplepositive,
  author        = {Checa, C. and Feliu, E.},
  title         = {An effective criterion for multiple positive zeros of vertically parametrized polynomial systems},
  year          = {2025},
  howpublished  = {\texttt{arXiv:2512.07560}},
  eprint        = {2512.07560},
  archivePrefix = {arXiv}
}

@misc{checa2026toricvectorbundlesnumber,
  author        = {Checa, C.},
  title         = {Toric vector bundles and the number of zeros of vertical systems},
  year          = {2026},
  howpublished  = {\texttt{arXiv:2608.27019}},
  eprint        = {2608.27019},
  archivePrefix = {arXiv}
}

@book{coxlitosh,
  author    = {Cox, D. and Little, J. and O'Shea, D.},
  title     = {Using Algebraic Geometry},
  volume    = {185},
  publisher = {Springer},
  year      = {2015}
}

@article{CURRAN2007115,
  author  = {Curran, R. and Cattani, E.},
  title   = {Restriction of {A}-discriminants and dual defect toric varieties},
  journal = {J. Symbolic Comput.},
  volume  = {42},
  number  = {1},
  pages   = {115--135},
  year    = {2007}
}

@article{dandrea2002,
  author  = {D'Andrea, C.},
  title   = {{M}acaulay style formulas for sparse resultants},
  journal = {Trans. Amer. Math. Soc.},
  volume  = {354},
  year    = {2001}
}

@article{dandrea2025sparse,
  author  = {D'Andrea, C. and Jeronimo, G.},
  title   = {Sparse {N}ullstellensatz, resultants, and determinants of complexes},
  journal = {Int. Math. Res. Not.},
  volume  = {2025},
  number  = {12},
  pages   = {rnaf174},
  year    = {2025}
}

@article{dandrea2020cannyemiris,
  author  = {D'Andrea, C. and Jeronimo, G. and Sombra, M.},
  title   = {The {C}anny--{E}miris conjecture for the sparse resultant},
  journal = {Found. Comput. Math.},
  volume  = {23},
  pages   = {1--61},
  year    = {2022}
}

@incollection{dalbec1995introduction,
  author    = {Dalbec, J. and Sturmfels, B.},
  title     = {Introduction to {C}how forms},
  booktitle = {Invariant Methods in Discrete and Comput. Geom.},
  pages     = {37--58},
  publisher = {Kluwer Academic Publishers},
  address   = {Dordrecht},
  year      = {1995}
}

@article{precision,
  author  = {de Greef, T. and Masroor, S. and Peletier, M. and Pendavingh, R.},
  title   = {Precision and sensitivity in detailed-balance reaction networks},
  journal = {SIAM J. Appl. Math.},
  volume  = {76},
  pages   = {2123--2153},
  year    = {2016}
}

@incollection{dickenstein2016biochemical,
  author    = {Dickenstein, A.},
  title     = {Biochemical reaction networks: an invitation for algebraic geometers},
  booktitle = {Mathematical Congress of the Americas},
  series    = {Contemp. Math.},
  volume    = {656},
  pages     = {65--83},
  year      = {2016}
}

@article{iterateddiscriminants,
  author  = {Dickenstein, A. and Di Rocco, S. and Morrison, R.},
  title   = {Iterated and mixed discriminants},
  journal = {J. Comb. Algebra},
  volume  = {7},
  pages   = {45--81},
  year    = {2023}
}

@article{dickenstein2002elimination,
  author  = {Dickenstein, A. and Sturmfels, B.},
  title   = {Elimination theory in codimension two},
  journal = {J. Symbolic Comput.},
  volume  = {34},
  number  = {2},
  pages   = {119--135},
  year    = {2002}
}

@article{diroccodefect,
  author  = {Di Rocco, S.},
  title   = {Projective duality of toric manifolds and defect polytopes},
  journal = {Proc. London Math. Soc.},
  volume  = {93},
  year    = {2006}
}

@article{tropicaldiscriminants,
  author  = {Dickenstein, A. and Feichtner, E. M. and Sturmfels, B.},
  title   = {Tropical discriminants},
  journal = {J. Amer. Math. Soc.},
  volume  = {20},
  year    = {2005}
}

@misc{esterov2026engineeredcompleteintersectionsalgorithmic,
  author        = {Esterov, A. and Mohr, R. and Mukhina, Y.},
  title         = {Engineered complete intersections: algorithmic aspects},
  year          = {2026},
  howpublished  = {\texttt{arXiv:2607.23622}},
  eprint        = {2607.23622},
  archivePrefix = {arXiv}
}

@article{esterov2010newton,
  author  = {Esterov, A.},
  title   = {{N}ewton polyhedra of discriminants of projections},
  journal = {Discrete Comput. Geom.},
  volume  = {44},
  pages   = {96--148},
  year    = {2008}
}

@misc{esterov2024engineeredcompleteintersectionsslightly,
  author        = {Esterov, A.},
  title         = {Engineered complete intersections: slightly degenerate {B}ernstein--{K}ouchnirenko--{K}hovanskii},
  year          = {2024},
  howpublished  = {\texttt{arXiv:2401.12099}},
  eprint        = {2401.12099},
  archivePrefix = {arXiv}
}

@book{feinberg-book,
  author    = {Feinberg, M.},
  title     = {Foundations of Chemical Reaction Network Theory},
  publisher = {Springer},
  year      = {2019}
}

@misc{feliu2025toricinvarianceverticallyparametrized,
  author        = {Feliu, E. and Henriksson, O.},
  title         = {Toric invariance of vertically parametrized systems},
  year          = {2025},
  howpublished  = {\texttt{arXiv:2411.15134}},
  eprint        = {2411.15134},
  archivePrefix = {arXiv}
}

@misc{feliu2026rootboundsverticalsystems,
  author        = {Feliu, E. and Helminck, P. A. and Henriksson, O. and Ren, Y. and Schr{\"o}ter, B. and Telek, M. L.},
  title         = {Root bounds of vertical systems using tropical geometry},
  year          = {2026},
  howpublished  = {\texttt{arXiv:2605.07645}},
  eprint        = {2605.07645},
  archivePrefix = {arXiv}
}

@article{FELIU2025630,
  author  = {Feliu, E. and Henriksson, O. and Pascual-Escudero, B.},
  title   = {Generic consistency and nondegeneracy of vertically parametrized systems},
  journal = {J. Algebra},
  volume  = {677},
  pages   = {630--666},
  year    = {2025}
}

@misc{feliu2025genericgeometrysteadystate,
  author        = {Feliu, E. and Henriksson, O. and Pascual-Escudero, B.},
  title         = {The generic geometry of steady state varieties},
  year          = {2025},
  howpublished  = {\texttt{arXiv:2412.17798}},
  eprint        = {2412.17798},
  archivePrefix = {arXiv}
}

@article{10.1007/s10801-018-0816-4,
  author  = {Forsg{\aa}rd, J.},
  title   = {Defective dual varieties for real spectra},
  journal = {J. Algebraic Comb.},
  volume  = {49},
  number  = {1},
  pages   = {49--67},
  year    = {2019},
  doi     = {10.1007/s10801-018-0816-4}
}

@article{localglobal,
  author  = {Feliu, E. and Pascual-Escudero, B.},
  title   = {Local and global robustness at steady state},
  journal = {Math. Methods Appl. Sci.},
  volume  = {45},
  number  = {1},
  pages   = {359--382},
  year    = {2022}
}

@book{gkz1994,
  author    = {Gelfand, I. M. and Kapranov, M. M. and Zelevinsky, A. V.},
  title     = {Discriminants, Resultants, and Multidimensional Determinants},
  publisher = {Birkh{\"a}user},
  address   = {Boston, MA},
  year      = {1994}
}

@misc{helmincktropical,
  author        = {Helminck, P. A. and Henriksson, O. and Ren, Y.},
  title         = {A tropical method for solving parametrized polynomial systems},
  year          = {2024},
  howpublished  = {\texttt{arXiv:2409.13288}},
  eprint        = {2409.13288},
  archivePrefix = {arXiv}
}

@article{genericFlatness,
  author  = {Helminck, P. A. and Ren, Y.},
  title   = {Generic root counts and flatness in tropical geometry},
  journal = {J. Lond. Math. Soc.},
  volume  = {111},
  number  = {5},
  pages   = {e70171},
  year    = {2025}
}

@article{HoltzRon2011,
  author  = {Holtz, O. and Ron, A.},
  title   = {Zonotopal algebra},
  journal = {Adv. Math.},
  volume  = {227},
  number  = {2},
  pages   = {847--894},
  year    = {2011}
}

@article{Kaihnsa_Nguyen_Shiu_2024,
  author  = {Kaihnsa, N. and Nguyen, T. and Shiu, A.},
  title   = {Absolute concentration robustness and multistationarity in reaction networks: conditions for coexistence},
  journal = {Eur. J. Appl. Math.},
  volume  = {35},
  number  = {4},
  pages   = {566--600},
  year    = {2024}
}

@misc{kaveh2025vectorvaluedlaurentpolynomialequations,
  author        = {Kaveh, K. and Khovanskii, A. and Spink, H.},
  title         = {Vector-valued {L}aurent polynomial equations, toric vector bundles and matroids},
  year          = {2025},
  howpublished  = {\texttt{arXiv:2507.09793}},
  eprint        = {2507.09793},
  archivePrefix = {arXiv}
}

@article{khovanskiitoroidal,
  author  = {Khovanskii, A.},
  title   = {Newton polyhedra and toroidal varieties},
  journal = {Funct. Anal. Appl.},
  volume  = {11},
  pages   = {289--296},
  year    = {1977}
}

@article{Khovanskii2016NewtonPA,
  author  = {Khovanskii, A.},
  title   = {Newton polytopes and irreducible components of complete intersections},
  journal = {Izv. Math.},
  volume  = {80},
  pages   = {263--284},
  year    = {2016}
}

@inproceedings{mantzaflaris:hal-01528377,
  author    = {Mantzaflaris, A. and Schost, {\'E}. and Tsigaridas, E.},
  title     = {Sparse rational univariate representation},
  booktitle = {Proceedings of ISSAC '17},
  address   = {Kaiserslautern, Germany},
  year      = {2017}
}

@book{oxley,
  author    = {Oxley, J. G.},
  title     = {Matroid Theory},
  publisher = {Oxford University Press},
  address   = {USA},
  year      = {2006}
}

@misc{rodriguezpajares2025combinatorialdescriptionselfassociatedset,
  author        = {Rodr{\'i}guez-Pajares, G. and Ruano, D. and Salizzoni, F.},
  title         = {A combinatorial description of when a self-associated set of points fails to be arithmetically {G}orenstein},
  year          = {2025},
  howpublished  = {\texttt{arXiv:2512.16766}},
  eprint        = {2512.16766},
  archivePrefix = {arXiv}
}

@book{sottile2011real,
  author    = {Sottile, F.},
  title     = {Real Solutions to Equations from Geometry},
  series    = {Univ. Lecture Ser.},
  volume    = {57},
  publisher = {American Mathematical Society},
  address   = {Providence, RI},
  year      = {2011}
}

@article{sturmfels94,
  author  = {Sturmfels, B.},
  title   = {On the {N}ewton polytope of the resultant},
  journal = {J. Algebraic Combin.},
  volume  = {3},
  pages   = {207--236},
  year    = {1994}
}

\Addresses

\end{document}